\documentclass[twoside,12pt]{amsart}
\usepackage[mathscr]{eucal}
\usepackage{amsmath,amsfonts}
\usepackage[dvips]{graphicx}
\usepackage{color}
\usepackage{amsthm}
\usepackage{amssymb}
\usepackage{latexsym, bbm, bm, subfigure, pict2e, multirow}
\usepackage{etoolbox}
\usepackage{epstopdf}
\preto\subequations{\ifhmode\unskip\fi}
\allowdisplaybreaks
 
\def\Xint#1{\mathchoice
{\XXint\displaystyle\textstyle{#1}}%
{\XXint\textstyle\scriptstyle{#1}}%
{\XXint\scriptstyle\scriptscriptstyle{#1}}%
{\XXint\scriptscriptstyle\scriptscriptstyle{#1}}%
\!\int}
\def\XXint#1#2#3{{\setbox0=\hbox{$#1{#2#3}{\int}$ }
\vcenter{\hbox{$#2#3$ }}\kern-.58\wd0}}

\def\dashint{\Xint-}

\newtheorem{theorem}{Theorem}[section]
\newtheorem{lemma}[theorem]{Lemma}

\numberwithin{equation}{section}
\newtheorem{remark}{Remark}
\numberwithin{remark}{section}

\graphicspath{{figures/}}

\def\tb{\textcolor{blue}}

\begin{document}
\title[Positivity-preserving DDG schemes for PNP equations]{Positivity-preserving third order DG schemes for Poisson--Nernst--Planck equations}
\author[H.~Liu, Z. Wang, P. Yin, H. Yu]{Hailiang Liu$^\dagger$, Zhongming Wang$^\ddagger$,  Peimeng Yin$^\S$ and Hui Yu$^\mathparagraph$ \\  }
\address{$^\dagger$Iowa State University, Mathematics Department, Ames, IA 50011, USA} \email{hliu@iastate.edu}
\address{$^\ddagger$ Florida International University,  Department of Mathematics and Statistics,  Miami, FL 33199, USA}
\email{zwang6@fiu.edu}
\address{$^\S$ Wayne State University,  Department of Mathematics, Detroit, MI 48202, USA}
\email{pyin@wayne.edu}
\address{$^\mathparagraph$ Yau Mathematical Sciences Center, Tsinghua University; Yanqi Lake Beijing Institute of Mathematical Sciences and Applications, Beijing, 100084, China} 
\email{huiyu@tsinghua.edu.cn}

\subjclass{35K51, 65M60}
\keywords{Poisson-Nernst-Planck system, positivity, direct discontinuous Galerkin methods}
\begin{abstract} In this paper, we design and analyze third order positivity-preserving discontinuous Galerkin (DG) schemes for solving the time-dependent system of Poisson--Nernst--Planck (PNP) equations, which has found much use in diverse applications. Our DG method with Euler forward time discretization is shown to preserve the positivity of cell averages at all time steps. The positivity of numerical solutions is then restored by a scaling limiter in reference to positive weighted cell averages. The method is also shown to preserve steady states. Numerical examples are presented to demonstrate the third order accuracy and illustrate the positivity-preserving property in both one and two dimensions.
\end{abstract}

\maketitle

\section{Introduction}
In this paper, we propose  
a positivity-preserving third order discontinuous Galerkin (DG) method to solve the Poisson--Nernst--Planck (PNP) system:  
\begin{subequations}\label{mc}
\begin{align}
\partial_t c_i & = \nabla \cdot(\nabla c_i+ q_i c_i \nabla \psi),   \\
- \Delta \psi & = \sum_{i=1}^m q_i c_i  +\rho_0(x),
\end{align}
\end{subequations}
where $c_i=c_i(t, x) $ is the local concentration of $i^{th}$ charged molecular or ion species with charge $q_i$ ($1\leq i\leq m$),  $\psi=\psi(t, x)$ is the electrostatic potential governed by the Poisson equation with $\rho_0$ as the fixed charge. The PNP system has been widely used to describe drift and diffusion phenomena for a variety of devices, including  biological ion channels \cite{Ei98} and  semiconductor devices \cite{MRS90}. 
However, the PNP system is nonlinear and strongly coupled, efficient computation of it is highly non-trivial. One main challenge for developing numerical schemes for (\ref{mc}) is to ensure density positivity (or non-negativity), since negative ion concentrations would be non-physical.



The solution to (\ref{mc}) is known to have three main properties:  mass conservation (when subject to zero flux boundary conditions), non-negativity of density, and free energy dissipation. These intrinsic solution features are naturally desired for any numerical algorithm to solve this system. 
The first property requires the scheme to be conservative. The second property which is also necessary for the third property is most difficult to achieve since it is point-wise. Numerical schemes addressing both density positivity and energy dissipation have been intensively studied. 
This is evidenced by recent results in \cite{FMEKLL14, HP16, LW14} with second order finite difference schemes.  
Based on some formulations of the nonlogarithmic Landau type (see (\ref{nonlog}) below), the semi-implicit schemes in  \cite{DWZ20, HP19, HH20, LW20a, LW20b} have been shown to feature unconditional positivity, while the energy dissipation is handled differently.  For instance, the unconditional positive schemes in \cite{LW20b} are linear, and shown to feature energy dissipation with only an $O(1)$ time step restriction. However, the spacial discretization in all these works is limited to only second order.

The authors in \cite{LW17} presented an arbitrary high order DG method for \eqref{mc}, that naturally incorporates both mass conservation and free energy dissipation properties. Their scheme achieves high order accuracy but the provable positivity-preserving property only holds for certain cases. The DG method is a class of finite element methods, using a completely discontinuous piecewise polynomial space for the numerical solution and the test functions.  
More general information about DG methods for elliptic, parabolic, and hyperbolic PDEs can be found in the recent books and lecture notes (see, e.g., \cite{HW07,KWLK00, Ri08, Shu09})

The main purpose of this paper is to propose a direct DG (DDG) method for (\ref{mc}), which is of third order in space and features a provable positivity-preserving property.  
The DDG method is a special class of the DG methods introduced in \cite{LY09, LY10} specifically for diffusion. Its key feature lies in numerical flux choices for the solution gradient, which involves interface jumps of both the solution against parameter $\beta_0$, and the second order derivatives against parameter $\beta_1$.  With an admissible condition of form  $\beta_0 > \Gamma_d(k,\beta_1)$, the DDG schemes are provably $L^2$ stable and optimally convergent \cite{Liu15, L21} as well as superconvergent for $\beta_1\not=0$ \cite{CLZ17}. The DDG method has been successfully applied to various application problems, including  linear and nonlinear Poisson equations \cite{HLY12, YHL14, YHL18} and Fokker-Placnk type equations \cite{LW16, LW17, LY14b, LY15}. In this paper, we use such DDG method for solving the Poisson equation (see Section 3.3).  

For the NP equation, our idea is to use the nonlogarithmic Landau transformation
\begin{equation}\label{nonlog}
c_i=g_ie^{-q_i\psi},
\end{equation}
so that (\ref{mc}a) reduces to 
$$
\partial_t c_i =\nabla \cdot \left( M_i \nabla g_i \right),  \quad c_i  =g_iM_i, \quad M_i=e^{-q_i \psi}, 
$$
to which we apply the DDG spatial discretization.
Here we shall identify the admissible range for the pair $(\beta_0, \beta_1)$ to ensure the positivity-preserving property for density functions $c_i$.

The two main ingredients in our schemes are:  
\begin{itemize}
\item[(i)] A positive decomposition over a test set of three points
for the weighted numerical integration;
\item[(ii)] Positivity-preserving limiter with Euler (or high order SSP-RK) time discretization. The limiter is based on weighted cell averages.  
\end{itemize}
As for (i), the test set is used to stabilize the scheme in $L^\infty$. The use of the DDG method is essential for identifying a test set of form (in one dimensional setting with uniform meshes of size $h$),  
$$
S_j = x_j +\frac{h}{2}\{-1, \gamma, 1\},
$$
where the existence of $\gamma \in (-1, 1)$ is ensured by the DDG method with parameters satisfying  
$$
\frac{1}{8} \leq \beta_1 \leq \frac{1}{4}, \quad \beta_0\geq 1. 
$$
The rigorous justification for the existence of  $\gamma$ follows that in \cite{ LY14b} for the linear Fokker-Planck (FP) equation.  

As for (ii), the projection  of $c_i=g_iM_i$ into the DG space allows to  transfer the positivity of cell averages of $c_i$ to the positivity of weighted cell averages of $g_i$, which will be used in defining the limiter.  

The implementation algorithm consists of several steps: (i) from $c_i$ we solve the Poisson equation by the DDG method to obtain $\psi$;  (ii) we further calculate $M_i=e^{-q_i \psi}$, and obtain $g_i$ by the projection of $c_i=g_iM_i$ to the DG space; (iii) we then solve NP equation $\partial_t c_i =\nabla \cdot \left( M_i \nabla g_i \right)$ to obtain $c_i$, while limiter is applied when necessary.  In other words, we first check the non-negativity of $g_i$ on $S_j$. If  negative values show up, the positivity-preserving limiter will be applied. We then apply the DDG method to solve the NP equation.  In 2D, the test set of points in each cell is constructed in a dimension by dimension manner.
  
Note that the established positivity-preserving property for density is independent of the Poisson solver, hence the numerical flux parameters for the Poisson equation can be different from those for the NP equation even within the same DDG framework. 

\subsection{Further related work} 
There is a considerable amount of literature that has been devoted to the numerical study of the PNP system. Many algorithms were introduced to handle specific issues in complex applications, in which one may encounter different numerical obstacles, such as discontinuous coefficients, singular charges,
and geometric singularities to accommodate various phenomena exhibited by biological ion channels; See, e.g., \cite{GNE04,  MZ14,ZCW11}. 

Recent efforts have been on the design of efficient and stable  methods with structure-preserving analysis. On regular domains, results using finite difference/volume for spatial discretization are quite rich, including the works \cite{DWZ20,FMEKLL14,GaoHe_JSC17, HP16, HP19, HH20, LW20a, LW20b,LW14,SX21}, as we discussed above. On irregular domains, Mirzadeh et al \cite{MG14} presented a conservative hybrid method with adaptive strategies. In \cite{PS09}, linearized finite element schemes that preserve electric energy decay and entropy decay properties were presented. A finite element method to the PNP system was introduced in \cite{MXL16} using a logarithmic transformation of the charge carrier densities, while the involved energy estimate resembles the physical energy law that governs the PNP system in the continuous case.  

A related and widely known model is the class of nonlinear Fokker--Planck equations
\begin{equation*}
\partial_t c =\nabla_x\cdot(f(c)\nabla_x (\psi(x) +H'(c))), 
\end{equation*}
where $f, H$ are some nonlinear functions, and the potential $\psi$ is given. For this model,  the high order DG method introduced in \cite{LW16} is shown to satisfy the discrete entropy dissipation law, extending the result for linear Fokker-Planck equations \cite{LY15}. A high order nodal DG scheme using Gauss-Lobatto quadrature  was developed in \cite{SCS18}, in which both  entropy dissipation and solution positivity are preserved by applying a limiter under a time step constraint, yet accuracy deterioration was observed in some test cases.

Another work quite relevant to ours is \cite{YL19}, in which the authors developed a third order positivity preserving DDG scheme for convection-diffusion equations with anisotropic diffusivity, while one main difficulty stems from the anisotropic diffusion.  
 
We now conclude this section by outlining the rest of this paper: in Section 2 we present the DDG method for the reformulated PNP system in one dimensional case, followed by the proof of positivity preservation in Section 3. We also discuss the positivity-preserving limiter and implementation details in Section 3. The two dimensional DDG scheme and the positivity-preserving analysis are presented in Section 4. Both one and two dimensional numerical examples are tested and results are reported in Section 5.  Finally we conclude the work in Section 6. 

\section{PNP model, reformulation and the DDG scheme}\label{SecForm}
\subsection{The PNP model} Let $\Omega$ be a bounded domain in $\mathbb{R}^d$, with $\textbf{n}$ being a unit exterior normal vector on the boundary $\partial \Omega$.
We consider the initial-boundary value problem
\begin{subequations}\label{PNP}
\begin{align}
&\partial_t c_i= \nabla\cdot(\nabla c_i+q_ic_i\nabla\psi), \quad  x\in  \Omega, \; t>0, \quad i=1,\cdots, m,\\
&-\Delta \psi =  \sum_{i=1} ^m q_i c_i + \rho_0,  \quad x\in  \Omega, \; t>0, \\
&c_i(0,x)=c_i^{\rm in}(x), \quad x\in \Omega;  \quad \frac{\partial c_i}{\partial  \textbf{n}}+q_ic_i \frac{\partial \psi}{\partial  \textbf{n}} =0  \mbox{~~on~}  \partial\Omega, \; t>0,\\
& \psi =\psi_D \mbox{~~on~} \partial\Omega_D,  \mbox{~~and~} \frac{\partial \psi}{\partial  \textbf{n}}  =\sigma  \mbox{~~on~} \partial\Omega_N, \quad t>0, 
\end{align}
\end{subequations}
where 
initial data $c^{\rm in}$ is given and zero flux boundary conditions are imposed. 
Here $\psi_D$ is the Dirichlet boundary condition, which models an applied  voltage; and $\sigma$ is the Neumann boundary condition, which models surface charge \cite{MXL16}.

\subsection{Reformulation of the PNP equations} We recall that an energy satisfying DG method based on the formulation 
\begin{align*}
\partial_t c_i & =\nabla \cdot \left( c_i \nabla p_i \right), \\
p_i & =q_i\psi+\log c_i,\\
-\Delta \psi & = \sum_{i=1}^mq_ic_i+\rho_0(x),
\end{align*}
was developed in \cite{LW17}, where the scheme is of arbitrarily high order, yet positivity of cell averages as needed for the limiting reconstruction is unwarranted.  In this work we  design a novel DG scheme by reformulating the PNP system (\ref{mc}) as
\begin{subequations}\label{mc+}
\begin{align}
\partial_t c_i & =\nabla \cdot \left( M_i \nabla g_i \right), \\
c_i & =g_iM_i, \quad M_i=e^{-q_i \psi},\\
-\Delta \psi & = \sum_{i=1}^mq_ic_i+\rho_0(x).
\end{align}
\end{subequations}
Here $g_i$ is calculated based on given $c_i$ and $M_i$. 

\subsection{The DDG scheme} \label{ddg} 
For simplicity, we consider the domain $\Omega$ to be a union of rectangular elements denoted by $\mathcal{T}_h = \{K\}$,
and $h$ denotes the mesh size of all the elements of $\mathcal{T}_h$.   
We set the DG finite element space as
$$
V_h=\{v\in L^2(\Omega): \forall K\in \mathcal{T}_h, \; v|_K \in P^k(K)\},
$$
where $P^k(K)$ is the space of polynomial functions of degree at most $k$  on $K$.  For quantities crossing interfaces, we need to  define both jumps and averages.
Let the set of the interior interfaces by $\Gamma^0$. Let the normal vector $n$ be assumed to be oriented from $K_1$ to $K_2$, sharing a common edge (face) $e \in \Gamma^0$, and  $h_e$ is the average of mesh sizes of the two neighboring cells in $n$ direction (or the mesh itself on boundary faces). We define the average $\{w\}$ and the jump $[w]$ of $w$ on $e$  as
$$
\{w\}=\frac{1}{2} (w|_{K_1}+w|_{K_2}), \quad [w]= w|_{K_2} - w|_{K_1} \quad \forall e\in \partial K_1\cap \partial K_2.
$$
 
With such approximation space, the semi-discrete DDG scheme  is to find $c_{ih}, g_{ih}, \psi_h \in V_h$ with $M_{ih}:=e^{-q_i \psi_h}$ such that for all $v, r, \eta \in V_h$, $i=1, \cdots, m$,
\begin{subequations}\label{dg+}
{\small
\begin{align}
 & \int_{K}\partial_t c_{ih} v\,dx
=  - \int_{K} M_{ih} \nabla g_{ih} \cdot\nabla v\,dx  + \int_{\partial K }\{M_{ih}\} \left( \widehat{\partial_n g_{ih}} v+ (g_{ih}-\{g_{ih}\})\partial_n v \right)\,ds,\\
& \int_{K} g_{ih}M_{ih} rdx =\int_{K} c_{ih} r dx,\\
 & \int_{K}  \nabla \psi_{h} \cdot \nabla \eta  dx - \int_{\partial K} \left( \widehat { \partial_n \psi_{h}} \eta   +(\psi_h - \{\psi_h\})\partial_n \eta \right)ds  = \int_{K}
  \left(\sum_{i=1}^m q_i c_{ih} +\rho_0\right)
 \eta dx,
\end{align}}
\end{subequations}
where $\widehat{\partial_n g_{ih}}=Fl_n(g_{ih})$ and $\widehat{\partial_n \psi_h}=Fl_n(\psi_h)$,
with the diffusive flux operator $Fl_n(\cdot)$ defined on the interface $e$ by
\begin{align*}
& Fl_n(w):=\beta_0 \frac{[w]}{h_e} +\{\partial_n w\}+\beta_1h_e[\partial_n^2 w].
\end{align*}
This is the DDG diffusive flux introduced in \cite{LY10} for diffusion. Here the parameters  $(\beta_0, \beta_1)$ are in the range to be specified so that the underlying scheme can satisfy certain positivity--principle.
Note that the DDG scheme with interface corrections as we use here was proposed in \cite{LY10} for the diffusion problem, as an improved version of the DG scheme in \cite{LY09}.

\subsection{Initial and boundary conditions for the DDG scheme}\label{IBC}
The initial data for $c_{ih}$ is generated by the piecewise $L^2$ projection, $c_{ih}(0,x)=\Pi c_i^{\rm in}(x)$, i.e.,
\begin{equation}\label{proj}
\int_{K} ( c_{ih}(0, x)  - c_i^{\rm in}(x)) v\,dx=0 \quad \forall v\in V_h.
\end{equation}
 The zero-flux boundary condition for $c_{i}$ can be weakly enforced through the boundary fluxes
 for $g_{ih}$ as 
 \begin{align*}
& Fl_n(g_{ih})=0,  \quad \{g_{ih}\}=g_{ih},\quad e\in \partial \Omega.
\end{align*}
For potential $\psi$ with the boundary data given in (\ref{PNP}d), the numerical fluxes on the boundary are defined as follows.
\begin{equation}\label{Dirichleta} 
\begin{aligned}
&\{\psi_h\}=\psi_D,\quad \widehat{ \partial_n \psi_{h}}=\frac{\beta_0}{h_e}(\psi_D-\psi_h) +\partial_\textbf{n} \psi_{h}, \quad e \in \partial \Omega_D,  \\
&\{\psi_h\}=\psi_h,\quad \widehat{ \partial_n \psi_{h}}=\sigma, \quad e \in \partial \Omega_N.
\end{aligned}
\end{equation}

\section{Positivity-preserving schemes in one dimension} \label{Sec1D} The semi-discrete scheme in Section \ref{SecForm}  is  complete if the parameter pair $(\beta_0, \beta_1)$ is admissible for ensuring the solution  positivity. In this section, we study the positivity preserving property of the DDG scheme with forward Euler discretization in one dimension.  
The two dimensions case will be presented in the next section. 

{Note that for $P^k$ polynomials with $k=0$, the DG formulation can lead to the scheme in \cite{LW14}; for $k=1$, second order schemes for $\beta_0\geq 1$ can be designed to preserve positive cell averages by following the techniques in \cite{LY14b}. Here we focus only on third order schemes that feature the positivity-preserving property.}

\subsection{Propagation of positive cell averages} 
We assume the interval $\Omega=\bigcup\limits_{j=1}^NI_j$, where $I_j = [x_{j-\frac12}, x_{j+\frac12}]$. For concise presentation, a uniform mesh $h=|I_j|$ is assumed.
We consider the first order Euler forward temporal discretization of (\ref{dg+}) to obtain
{\small \begin{align}\label{fully_1D}
\int_{I_j} \frac{c_{ih}^{n+1}-c_{ih}^n}{\Delta t} v\,dx & = -\int_{I_j} M_{ih}^n \partial_x g^n_{ih} \partial_x v\,dx + \left.\{M^n_{ih}\}\left[\widehat{\partial_x g^n_{ih}}v +(g^n_{ih} -\{g^n_{ih}\})\partial_x v\right]\right|^{x_{j+\frac12}}_{x_{j-\frac12}},
\end{align}
}
where $\Delta t>0$ is the time step, and
$$
M_{ih}^n=e^{-q_i \psi_h^n},
$$
where $\psi_h^n \in V_h$ is obtained from $c_{ih}^n$ by solving (\ref{dg+}c).  In \eqref{fully_1D} we used the following notation
$$\omega|^{x_{j+\frac12}}_{x_{j-\frac12}} =\omega(x^-_{j+\frac12})-\omega(x^+_{j-1/2}).$$

On interfaces $x_{j+1/2}$, $j=1, \cdots, N-1$, the numerical fluxes are chosen as
\begin{align*}
\widehat{\partial_x g_{ih}}& =\frac{\beta_0}{h}[g_{ih}] +\{\partial_x g_{ih}\} +\beta_1 h[\partial_x^2 g_{ih}], 
\end{align*}
and $\widehat{\partial_x g_{ih}}=0$, $\{g_{ih}\}=g_{ih}$ for $j=0, N$.

From $c_{ih}^n\in V_h$ and $M_{ih}^n$ at each time step, we obtain $g_{ih}^n$ by
\begin{align}\label{cg}
\int_{I_j}   g_{ih}^n M_{ih}^n r dx = \int_{I_j}c_{ih}^n r  dx \quad \forall r \in V_h.
\end{align}
By taking the test function $v = \frac{\Delta t}{h} $ in (\ref{fully_1D}) and $r = \frac{1}{h}$ in (\ref{cg}),
we obtain the evolutionary equation for the cell average,
\begin{equation*}\label{CD_cell_1D}
 \bar c^{n+1}_{ij}=  \langle g^n_{ih} \rangle + \mu h\left. \{M^n_{ih}\} \widehat{\partial_x g_{ih}^n}\right|^{x_{j+\frac12}}_{x_{j-\frac12}},
\end{equation*}
where
$$
 \langle g^n_{ih} \rangle :=  \frac{1}{h}\int_{I_j} g^n_{ih}M^n_{ih} dx=\bar c^n_{ij},
$$
and  $\mu := \frac{\Delta t}{h^2}$ is the mesh ratio.   In order to apply \cite[Theorem 3.4]{LY14b} we set
$$
M(x)=M_{ih}^n(x)=e^{-q_i\psi_h^n(x)}
$$ 
as a piecewise smooth weight, and use the notation
\begin{align*}
\langle \phi \rangle_j&=\frac{1}{2}\int_{-1}^1 \phi(\xi)M(x_j+\frac{h}{2}\xi)d\xi.
\end{align*}

We also define 
$$
a_j= \frac{\langle \xi -\xi^2\rangle_j }{\langle 1-\xi \rangle_j}, \quad b_j = \frac{\langle \xi +\xi^2\rangle_j }{\langle 1+\xi \rangle_j}.
$$
and 
\begin{align*}
\hat \omega_j^1(\gamma) & = \frac{\langle \gamma -\xi (1+\gamma) +\xi^2\rangle_j} {2(1+\gamma)}, \\
\hat \omega_j^2 (\gamma) & = \frac{\langle 1- \xi^2 \rangle_j}{1-\gamma^2}, \\
\hat \omega_j^3 (\gamma)& =\hat \omega^1(-\gamma).
\end{align*}
Hence for any $p\in P^2[-1, 1]$ we have the decomposition $$
 \langle p \rangle_j=\hat \omega_j^1(\gamma)p(-1)+\hat \omega_j^2(\gamma)p(\gamma)+\hat \omega_j^3(\gamma)p(1).
$$
{We recall the following key result (see also Lemma 2.1 in \cite{YL19})}. 
\begin{lemma}\cite[Lemma 3.3]{LY14b}  $\tilde{\omega}^i_j(\gamma) >0$ for $i=1, 2,3$ if and only if 
$$
\gamma \in (a_j, b_j),
$$
where $a_j, b_j$ satisfy $-1<a_j<b_j<1$. 
\end{lemma} 

{
\begin{remark}
In our numerical tests, $\gamma=\frac{1}{2}(a_j+b_j)$ is taken in each cell. Both $a_j$ and $b_j$ depends on $M_i$, hence the set $S_j$ may differ for each $g_i$. 
\end{remark}
}

We thus have the following result.
\begin{theorem}\label{thk2_1D}($k=2$) The scheme (\ref{fully_1D})-(\ref{cg}) with
\begin{align}\label{betak2_1D}
\frac{1}{8} \leq \beta_1  \leq \frac{1}{4} \quad \text{and} \quad  \beta_0\geq 1
\end{align}
is positivity preserving, namely, $\bar{c}_{ij}^{n+1} >0$  if $\bar{c}_{ij}^{n} >0$ and $g_{ih}^n(x) \geq 0  $ on the set $S_j$'s where
$$
S_j = x_j +\frac{h}{2}\left\{-1, \gamma, 1\right\}
$$
with $\gamma$ satisfying
\begin{align*}
a_j <\gamma <  b_j \quad \text{and} \quad |\gamma|\leq  8\beta_1-1,  
\end{align*}
under the CFL condition $\mu \leq \mu_0$, with  $M_{j+1/2}:=\{M^n_{ih}\}|_{x_{j+1/2}}$ in
\begin{align*}
  \mu_0 & =\min_{1\leq j\leq N}
  \left\{
  \frac{\hat \omega_j^1(\pm \gamma)}{\alpha_3(\mp \gamma)M_{j-1/2} +\alpha_1(\pm \gamma)M_{j+1/2}},
  \frac{(1-\gamma^2) \hat \omega_j^2}{2(1-4\beta_1) (M_{j-1/2} +M_{j+1/2})}
  \right\}, 
\end{align*}
where
$$
\alpha_1(\gamma)=\frac{8\beta_1-1+\gamma}{2(1+\gamma)}>0, \quad \alpha_3(\gamma)=\beta_0
+\frac{8\beta_1-3+\gamma}{2(1-\gamma)}>0.
$$
\end{theorem}
\begin{remark} The CFL conditions depend on $\psi^n_h$ 
due to the use of $M_{ih}^n=e^{-q_i\psi^n_h}$. They are sufficient conditions rather than necessary to preserve the positivity of solutions. Therefore, in practice, these CFL conditions are strictly enforced only in the case the positivity preserving property is violated.
\end{remark}
\begin{remark}
 The parameter range (\ref{betak2_1D}) was first identified in \cite{LY14b} for a third order DDG scheme to feature the maximum-principle-preserving  property for linear Fokker-Planck equations. 
\end{remark}

\subsection{Limiter}
Theorem \ref{thk2_1D} suggests that for the scheme with forward Euler discretization, we need to modify $g_{ih}^n$ using weight $M(x)=e^{-q_i \psi^n_h(x)}$ on $I_j$ such that it becomes non-negative on $S_j$. This can be done by using the following  scaling limiter. Let $w_h \in  P^k(I_j)$  be a high order approximation to a smooth function $w(x) \geq 0$, with cell averages $\bar{w}_j> 0$, where 
$$
\bar w_j :=\frac{\int_{I_j} M(x) w_h(x)dx}{\int_{I_j} M(x) dx}. 
$$
We then construct another polynomial by
\begin{equation}
\tilde{w}_h(x)= \bar{w}_j+\frac{\bar{w}_j}{\bar{w}_j-\min_{S_j} w_h(x)} (w_h(x)-\bar{w}_j),
\text{ where }
\theta = \min\left\{1,   \frac{\bar{w}_{j}}{\bar{w}_{j} -\min_{ S_{j}} w_h(x) } \right\}. \label{ureconstruct}
\end{equation}
This reconstruction maintains same cell averages and satisfies  $$\min_{S_j} \tilde{w}_h(x)\geq 0.$$
Moreover, it can be shown that if  $\bar w_j >0$, then  the above scaling limiter does not destroy the solution accuracy, as stated in the following lemma. 
\begin{lemma} {\cite[Lemma 3.5]{LY14b}} If $\bar w_j > 0$, then the modified polynomial $\tilde{w}_h$ is as accurate as $w$ in the following sense:
\begin{align*}
 |\tilde w_h(x)-w_h(x)| \leq C_k \| w_h-w\|_{\infty} \quad \forall x\in I_j,
 \end{align*}
where $C_k$ is a constant depending on the polynomial degree $k$.
\end{lemma}
The limiting techniques of this nature are inspired by the limiter introduced for conservation laws \cite{ZS10}; see also invariant region preserving limiters for systems of conservation laws in \cite{JL18, JL19}.

\subsection{DDG discretization for the Poisson problem} 
Note that the established positivity-preserving property for density is independent of the Poisson solver, hence the DDG flux parameters for the Poisson equation can be chosen independent of those for the NP equation. 

We now investigate the admissible parameters $(\beta_0, \beta_1)$ of the DDG scheme for the Poisson problem. Such scheme in its global formulation, obtained from the DDG scheme (\ref{dg+}c), is to find $\psi_h \in V_h$ so that
\begin{equation}\label{DDGPoisson}
    A(\psi_h, \eta)=L(c_{ih},\eta),\quad \forall \eta \in V_h,
\end{equation}
where $A(\psi_h, \eta)$ and $L(c_{ih}, \eta)$ are given by
\begin{equation*}
\begin{aligned}
A(\psi_h, \eta) = & \sum_{K\in \mathcal{T}_h}\int_K \nabla \psi_h \cdot \nabla \eta dx + \sum_{e\in \Gamma^0}\int_e \left( \widehat{\partial_n\psi_h}[\eta]+\{\partial_n\eta\}[\psi_h]\right)ds\\
& + \int_{\partial \Omega_D} \left(  \frac{\beta_0}{h_e}\psi_h - \partial_\textbf{n}\psi_h \right) \eta - \psi_h \partial_\textbf{n}\eta ds,
\end{aligned}
\end{equation*}
\begin{equation*}
\begin{aligned}
L(c_{ih},\eta) = & \sum_{K\in \mathcal{T}_h}\int_K  \left(\sum_{i=1}^m q_i c_{ih} +\rho_0\right)
 \eta dx + \int_{\partial \Omega_D} \left(  \frac{\beta_0}{h_e}\eta - \partial_\textbf{n}\eta \right) \psi_D ds + \int_{\partial \Omega_N} \sigma \eta ds.
\end{aligned}
\end{equation*}
It has been shown in \cite{L21} that there exists a number $\beta_0^*>1$ depending on $\beta_1,  k$ and the mesh geometry such that if $\beta_0>\beta_0^*$, then 
\begin{equation}\label{croc}
A(\eta,\eta) \geq \gamma \|\eta\|_E^2, \quad \eta \in V_h, 
\end{equation}
form some $\gamma \in (0,1)$. Here the energy norm is defined by  
\begin{equation*}
\|\eta\|_E = \left( \sum_{K\in \mathcal{T}_h}\int_K |\nabla \eta |^2 dx + \sum_{e\in \Gamma^0}\int_e \frac{1}{h_e} [\eta]^2ds + \int_{\partial \Omega_D} \frac{1}{h_e}\eta^2 ds \right)^{\frac{1}{2}}. 
\end{equation*} 
Note that for uniform rectangular meshes, we simply use $h$ for $h_e$.
A sufficient condition given in \cite{L21} is $\beta_0^* =
\Gamma_d(\beta_1)$ on all cell interfaces,  and $\beta_0^* = 2\Gamma_d(0)$ on the Dirichlet boundary faces $\subset \partial \Omega_D$.
For 1D and 2D rectangular meshes, the result in \cite{Liu15} gives  
$$
\Gamma_d(\beta_1) = k^2\left(1-\beta_1(k^2-1) + \frac{\beta_1^2}{3}(k^2-1)^2\right).
$$
Hence for the DDG scheme (\ref{DDGPoisson}) to the Poisson problem alone, we have the following result. 
\begin{theorem}\label{DDGStabThm} Suppose that  $\partial \Omega_D \not = \emptyset$. Given $c_{ih}$, if $\beta_0>\beta_0^*$, the DDG scheme (\ref{DDGPoisson}) admits a unique $\psi_h$. Moreover, let $\psi_h$ and $\tilde{\psi}_h$ be two solutions of (\ref{DDGPoisson}) corresponding to $c_{ih}$ and $\tilde{c}_{ih}$, respectively, then
\begin{equation}\label{DGstab}
\|\psi_h-\tilde{\psi}_h\| \leq C \sum_{i=1}^m |q_i|\|c_{ih}-\tilde{c}_{ih}\|    
\end{equation}
for some constant $C$ independent of $h$. 
\end{theorem}
\begin{proof} (\ref{DGstab}) when $c_{ih}=\tilde c_{ih}$ implies that there exists a unique solution $\psi_h$ for given $c_{ih}$.  We only need to prove  (\ref{DGstab}). 
The difference of (\ref{DDGPoisson}) with  $c_{ih}$ and $\tilde{c}_{ih}$ respectively, gives
$$
A(\psi_h-\tilde{\psi}_h, \eta) = \sum_{i=1}^m q_i(c_{ih}-\tilde{c}_{ih}, \eta).
$$
By taking $\eta=\psi_h-\tilde{\psi}_h \in V_h$ and applying (\ref{croc}), we have
$$
\gamma \|\eta\|^2_E \leq A(\eta,\eta) \leq \sum_{i=1}^m |q_i|\|c_{ih}-\tilde{c}_{ih}\|\|\eta\|.
$$
Since $\partial \Omega_D \not = \emptyset$, one can show that $ \sup_{w\in V_h,w\neq 0}  \frac{\|w\|}{\|w\|_E}$ is uniformly bounded from above, hence 
$$
\|\eta\|\leq C_0 \|\eta\|_E
$$
for some $C_0$. Combining these two inequalities, we obtain the claimed bound with $C=C_0^2/\gamma $.
\end{proof}
In connection to the coupled scheme for the NP system we make a few remarks.  
\begin{remark}
(i) A refined analysis similar to that in \cite{L21} would lead to the following error bound: 
\begin{align*}
\|\psi-\psi_h\| \leq C\left( h^{\min\{s, k+1\}} + \sum_{i=1}^m|q_i| \|c_i-c_{ih}\| \right),    
\end{align*} 
where $(c_i, \psi)$ solves the PNP system (\ref{PNP}) with $\psi\in H^s$, and $\psi_h$ is the solution to the DDG scheme (\ref{DDGPoisson}) based on polynomials of degree $k$.
Though we may use higher order DDG scheme to solve the Poisson problem, the error is limited by the error of the coupled scheme of order $3$ to the NP problem, so our method is only third order.\\
(ii) 
One may also apply the DDG method with the same flux parameters $(\beta_0, \beta_1)$ for both the Poisson equation and the NP equation as long as they  are taken from the following range:  
$$
\beta_0 \geq \beta_0^*, \quad \frac{1}{8} \leq \beta_1  \leq \frac{1}{4}.
$$
\end{remark}

{
\subsection{Preservation of steady-states} 
With no-flux boundary conditions, scheme (\ref{fully_1D}) can be shown to be steady-state preserving. A steady state for the PNP system is determined by 
$$
c_i=c_i^\infty e^{-q_i\phi},
$$
where  $c_i^\infty>0 $ is a constant, and $\phi$ solves the following Poisson-Boltzmann problem: 
\begin{subequations}\label{PB}  
\begin{align}
& -\Delta \phi=\sum_{i=1}^m q_i c_i^\infty e^{-q_i\phi}  +\rho_0(x), \\
& \phi =\psi_D \mbox{~~on~} \partial\Omega_D,  \mbox{~~and~} \frac{\partial \phi}{\partial  \textbf{n}}  =\sigma  \mbox{~~on~} \partial\Omega_N.
\end{align}
\end{subequations} 
The well-posedness for this problem may be established by a variational approach using the functional 
$$
G[\phi]=\int_{\Omega} \left[ \frac{1}{2}|\nabla \phi|^2 -\rho_0(x)\phi +\sum_{i=1}^m c_i^\infty (e^{-q_i \phi}-1) \right]dx-\int_{\partial \Omega_N }\sigma \phi ds 
$$
with the trial function space 
$$
\left\{\phi \in H^1(\Omega), \quad \phi =\psi_D \mbox{~~on~} \partial\Omega_D\right\}.
$$
For the case  $\partial \Omega_N=\emptyset$, we refer the reader to \cite[Theorem 2.1]{LiSIAMJMA09} (and also  \cite[Theorem 2.1]{Li11}) for a detailed account of existence results using the variational approach for both the point ions and the finite size ions. 

We say a discrete function $c_{ih}$ is at steady-state if
\begin{align}\label{ci}
  c_{ih}=c_i^\infty \Pi e^{-q_i \phi_h},
  \end{align}
where $\phi_h$ satisfies (\ref{dg+}c) for $\phi_h=\psi_h$ with $c_{ih}$ replaced by (\ref{ci}), which is a nonlinear algebraic equation for $\phi_h$. This may serve as a DDG scheme for the Poisson-Boltzmann problem (\ref{PB}). Instead of using this nonlinear scheme, we use our DDG scheme for PNP as an iterative scheme for obtaining $\phi_h$.    

We state the following nice property of the PNP scheme. 
\begin{theorem} 
Consider the fully discrete scheme (\ref{fully_1D}). If $c_{ih}^0$ is already at steady state, then $c_{ih}^n=c_{ih}^0$ for all $n\geq 1$ with any time step $\Delta t>0$.
\end{theorem}
\begin{proof} 
Take $c_{ih}^0= c_i^\infty \Pi e^{-q_i \psi_h}$ with $\psi_h:=\phi_h$, which means that for any $r\in V_h$,
$$
\int_K c_{ih}^0 rdx = \int_K c_i^\infty e^{-q_i \psi_h} rdx=\int_K c_i^\infty M_{ih} rdx, 
$$
which when combined with (\ref{dg+}b) implies
\begin{equation*}
\int_K g_{ih}^0 M_{ih} rdx = \int_K c_i^\infty M_{ih} rdx.
\end{equation*}
Since $g_{ih}^0 \in V_h$, hence we can take $r=g_{ih}^0- c_i^\infty$ to obtain 
$$
\int_K (g_{ih}^0-c_i^\infty)^2 M_{ih}dx=0.
$$
Hence $g_{ih}^0 \equiv c_i^\infty$. Insertion of this into the right of (\ref{fully_1D}) for $n=0$ gives 
$$
c_{ih}^1=c_{ih}^0. 
$$
Same induction ensures the asserted $c_{ih}^n = c_{ih}^0 $ for any $n\geq 1$. Here the size of time steps plays no role. 
\end{proof}
}

\subsection{Implementation details}\label{dgalg} 
We summarize our algorithm in following steps.
\begin{itemize}
\item[1.] (Initialization) Project $c_i^{\rm in}(x)$ onto $V_h$, as formulated in (\ref{proj}),  to obtain $c_{ih}^{0}(x)$.
\item[2.] (Poisson solver) From $c_i=c_{ih}^n$ Solve (\ref{dg+}c) to obtain $\psi=\psi_h^n$ subject to the boundary conditions \eqref{Dirichleta}.
\item[3.] (Projection)  Solve (\ref{dg+}b) to obtain $g_{ih}^n$.
\item[4.] (Reconstruction) Apply, if necessary, the scaling limiter \eqref{ureconstruct} on $g_{ih}^n$ to ensure that in each cell $g_{ih}^n \geq 0$ on $S_j$. 
\item[5.] (NP equation)  Solve  the fully discrete equation \eqref{fully_1D} to update $c_{ih}^{n+1}$ with second order Runge-Kutta (RK) ODE solver.
\item[6.] Repeat steps 2-5 until final time $T$.
\end{itemize}
\begin{remark}
The forward Euler time discretization in (\ref{fully_1D}) can be extended to higher order SSP Runge-Kutta method \cite{GST01}, which is a convex linear combination of the forward Euler method. The desired positivity preserving property is then ensured under a suitable CFL condition (see, e.g., \cite{ZS10}). 
We chose to use second order time discretization since it is sufficient not to outweigh the spatial error that comes from our third order DDG schemes.
\end{remark}
{
\begin{remark}
 The time step restriction $\Delta t \sim O(h^2)$ is obviously a drawback of the explicit time discretization. Usually one would use implicit in time discretization for diffusion and explicit time discretization for the nonlinear drift term (called IMEX in the literature) so that the time step restriction could be relaxed. Unfortunately, formulation (\ref{mc+}) does not support such a separation.
\end{remark}
}

\section{Positivity-preserving schemes in two dimensions} \label{SecMD}
In this section we extend our result to the two dimensional case.
\subsection{Scheme formulation with rectangular meshes}
Let the two dimensional domain $\Omega =[0, L_x]\times [0, L_y]$ be partitioned by uniform rectangular meshes  so that $\Omega=\cup I_{jl}$ with
$$
I_{jl}=[x_{j-\frac12}, x_{j+\frac12}]\times [y_{l-\frac12}, y_{l+\frac12}],
\quad  1\leq j\leq P, \; 1\leq l\leq Q,
$$
where $P$ and $Q$ are two positive integers, the mesh sizes $\Delta x= \frac{L_x}{P}$, $\Delta y= \frac{L_y}{Q}$ and
\[
x_{j+\frac12} = j\Delta x, \qquad y_{l+\frac12} = l\Delta y, \quad 0\leq j \leq P, \quad 0\leq l \leq Q.
\]
Consider the NP equation of form  
$$
\partial_t c
	= \partial_x (M \partial_x g) +\partial_y(M\partial_y g), \quad (x, y)\in \Omega \subset \mathbb{R}^2,
$$
subject to the initial condition $c(0, x, y)=c^{\rm in}(x, y)$  and  zero-flux boundary condition. Here $(c, g, M)=(c_i, g_i, M_i)$ for $i=1, \cdots, m$. The fully discrete scheme with Euler forward  discretization and the DDG spatial discretization is as follows: 
\begin{align}\label{fully2D}
\int_{I_{jl}}c_{h}^{n+1}v\,dxdy
=& \int_{I_{jl}}c_h^nv\,dxdy -\Delta t\int_{I_{jl}} M_{ih}^n\nabla g_h^n\cdot\nabla v\,dxdy \notag\\
&+ \Delta t\left.\int_{y_{l-\frac12}}^{y_{l+\frac12}}\{M_{ih}^n\} \left[\widehat{\partial_x g_h^n}v +(g_h^n -\{g_h^n\})\partial_x v\right]\,dy\right|_{x_{j-\frac12}}^{x_{j+\frac12}} \nonumber\\
&+ \Delta t\left.\int_{x_{j-\frac12}}^{x_{j+\frac12}}\{M_{ih}^n\}\left[\widehat{\partial_y g_h^n}v +(g_h^n -\{g_h^n\})\partial_y v\right]\,dx\right|_{y_{l-\frac12}}^{y_{l+\frac12}},
\end{align}
where
\begin{align*}
\left.\widehat{\partial_x g_h}\right|_{(x_{j+\frac12},y)}& =\frac{\beta_0}{\Delta x}[g_h] +\{\partial_x g_h\} +\beta_1 \Delta x [\partial_x^2 g_h],\\
\left.\widehat{\partial_y g_h}\right|_{(x,y_{l+\frac12})}& =\frac{\beta_0}{\Delta y}[g_h] +\{\partial_y g_h\} +\beta_1 \Delta y [\partial_y^2 g_h],
\end{align*}
on interior interfaces, and on the boundary they are zero and $\{g_h^n\}=g_h^n$.

Similar to the one-dimensional case, we introduce the cell average as
\begin{equation*}
\bar c_{jl} = \frac{\int_{I_{jl}}c_h\,dxdy}{\Delta x \Delta y}
	= \dashint_{x_{j-\frac12}}^{x_{j+\frac12}} \dashint_{y_{l-\frac12}}^{y_{l+\frac12}}c_h\,dxdy,
\end{equation*}
where $\dashint$ denotes the average integral. We obtain the cell average update from  (\ref{fully2D}) with $v=1$, divided by $\Delta x\Delta y$
\begin{align}\label{EF2Dv1}
\bar c^{n+1}_{jl} =& \bar c^n_{jl}
+ \mu_x\Delta x \left.\dashint_{y_{l-\frac12}}^{y_{l+\frac12}}\{M_{ih}^n\}\widehat{\partial_x g_h^n}\,dy\right|_{x_{j-\frac12}}^{x_{j+\frac12}}
+ \mu_y\Delta y \left.\dashint_{x_{j-\frac12}}^{x_{j+\frac12}}\{M_{ih}^n\}\widehat{\partial_y g_h^n}\,dx\right|_{y_{l-\frac12}}^{y_{l+\frac12}},
\end{align}
where $\mu_{x} = \frac{\Delta t}{(\Delta x)^2}$ and $\mu_{y} = \frac{\Delta t}{(\Delta y)^2}$.
Let $\mu= \mu_x + \mu_y$ and decompose $\bar c^n_{jl}$ as
\[
\bar c^n_{jl} = \frac{\mu_x}{\mu}\bar c^n_{jl} + \frac{\mu_y}{\mu}\bar c^n_{jl}=\frac{\mu_x}{\mu}\langle g^n_{h}\rangle_{jl} + \frac{\mu_y}{\mu}\langle g^n_{h}\rangle_{jl}, \quad \Delta x\Delta y \langle g \rangle_{jl}=\int_{I_{jl}} M_{ih}^n g(x)dx,
\]
so that (\ref{EF2Dv1}) can be rewritten as
\begin{align}\label{EF2DH}
 \bar c^{n+1}_{jl}  = \frac{\mu_x}{\mu}\dashint_{y_{l-\frac12}}^{y_{l+\frac12}}H_1(y)\,dy +\frac{\mu_y}{\mu}\dashint_{x_{j-\frac12}}^{x_{j+\frac12}} H_2(x)\,dx,
\end{align}
where
\begin{align*}
H_1(y) & =\langle g_h^n\rangle_j (y)
	+ \mu \Delta x\left. \{M_{ih}^n\}\widehat{\partial_x g_{h}^n}\right|_{x_{j-\frac12}}^{x_{j+\frac12}},\\
H_2(x) & =\langle g_h^n\rangle_l (x)
	+ \mu \Delta y\left. \{M_{ih}^n\}\widehat{\partial_y g_{h}^n}\right|_{y_{l-\frac12}}^{y_{l+\frac12}},
\end{align*}
here we have used the notation
$$
\langle g\rangle_j (y)=\frac{1}{\Delta x}\int_{x_{j-1/2}}^{x_{j+1/2}}g(x, y)M_{ih}^n(x,y)dx, \quad
\langle g\rangle_l (x)=\frac{1}{\Delta y}\int_{y_{l-1/2}}^{y_{l+1/2}}g(x, y)M_{ih}^n(x,y)dy.
$$
The two integrals in (\ref{EF2DH}) can be approximated by quadratures with sufficient accuracy. Let us assume that we use a Gauss quadrature with $L\geq \frac{k+2}{2}$ points, which has accuracy of at least $O(h^{k+2})$ with $h=\max\{\Delta x, \Delta y\}$.
 Let
\begin{equation*}
S_{j}^x = \{x_j^{\sigma}, \sigma = 1, \ldots, L\} \qquad \text{ and } \qquad
S_{l}^y = \{y_l^{\sigma}, \sigma = 1, \ldots, L\}
\end{equation*}
denote the quadrature points on $[x_{j-\frac12}, x_{j+\frac12}]$ and $[y_{l-\frac12},y_{l+\frac12}]$,  respectively.  The superscript $\sigma$ will denote the index of the Gauss quadrature points and  $\omega^{\sigma}$'s are the quadrature weights at the quadrature points, so that
$
\sum_{\sigma=1}^L \omega^{\sigma}=1.
$
Using the quadrature rule on the right-hand side of (\ref{EF2DH}), we obtain the following
\begin{equation}\label{EF2DHQ}
\bar c^{n+1}_{jl}= \frac{\mu_x}{\mu}\sum_{\sigma=1}^L \omega^{\sigma}H_1(y_l^\sigma) + \frac{\mu_y}{\mu}\sum_{\sigma=1}^L \omega^{\sigma}H_2(x_j^\sigma).
\end{equation}
Applying the one-dimensional result in Theorem \ref{thk2_1D} to both $H_1(y_l^\sigma)$ and $H_2(x_j^\sigma)$, we can establish the positivity-preserving  result for the two-dimensional case.  Let
\begin{equation*}
\hat S^x_j = x_j + \frac{\Delta x}{2}\{-1, \gamma^x, 1\} \qquad \text{ and } \qquad
\hat S^y_l = y_l + \frac{\Delta y}{2}\{-1, \gamma^y, 1\},
\end{equation*}
denote the test set on $[x_{j-\frac12}, x_{j+\frac12}]$ and $[y_{l-\frac12}, y_{l+\frac12}]$, respectively,  with $\gamma$ satisfying
\begin{align*} 
a_j <\gamma^x<b_j, \quad a_l <\gamma^y<b_l,  \quad |\gamma^x, \gamma^y |\leq 8\beta_1-1.
\end{align*}
Note that $a_j, b_j, \gamma^x$ depend on $y_l^\sigma$, and $a_l, b_l, \gamma^y$ depend on $x_j^\sigma$, through the weights
$M(x, y^\sigma_l)$ and $M(x_{j}^\sigma, y)$, respectively.   We use $\otimes$ to denote the tensor product and define
\begin{equation*}
S_{jl}=(S_j^x\otimes\hat S^y_l)\cup (\hat S_j^x \otimes S_l^y).
\end{equation*}
\begin{theorem}\label{th2dk2}($k=2$)
Consider the two dimensional  DDG method  (\ref{fully2D}) on rectangular meshes, associated with the approximation DG  polynomials $c^n_h(x, y)$ of degree $k$,  with $(\beta_0, \beta_1)$  chosen so that
$$
\frac{1}{8}  \leq  \beta_1 \leq \frac{1}{4}  \qquad\text{ and } \qquad \beta_0 \geq 1.
$$
If  $g_h^n(x, y) \geq 0 $ for all $(x, y)\in S_{jl}$ and $\bar c^{n}_{jl}>0$,  then $\bar c^{n+1}_{jl} >0 $ under the CFL condition
\[
\mu < \mu_0,
\]
where $\bar c^{n+1}_{jl}$ is given in (\ref{EF2DHQ}), $\mu_0$ is given in (\ref{lambda02D}) below.  
\end{theorem}

\begin{proof}
It is easy to check that  $\bar c^{n+1}_{jl} $ in (\ref{EF2DHQ}) is a convex combination of  $H_1(y_l^\sigma)$ and $H_2(x_j^\sigma)$ for $\sigma=1, \cdots, L$;  hence
$\bar c^{n+1}_{jl} > 0 $ if
$$
H_1(y_l^\sigma)> 0 \quad \text{and}\quad  H_2(x_j^\sigma)> 0, \quad \sigma=1, \cdots, L.
$$
Applying the one-dimensional result obtained in Theorem \ref{thk2_1D} to $H_1(y_l^\sigma)$,  we obtain that for each quadrature point $y\in S_l^y$,   $H_1(y)> 0$ if $g_h^n(x, y) \geq 0$
on the test set $\hat S_j^x$ and
$
\mu \leq \mu_0^x
$
with
\begin{align*}
\mu_0^x=\min_{jl}\min_{1\leq \sigma \leq L}\left\{
\frac{\langle\pm\gamma^x\mp\xi(1\pm\gamma^x)+\xi^2\rangle_j(y_l^\sigma)}{2(1\pm\gamma^x)\left(\alpha_3(\mp\gamma^x)M(x_{j-\frac12}, y_l^{\sigma}) + \alpha_1(\pm\gamma^x)M(x_{j+\frac12}, y_l^{\sigma})\right)} \right., \\
\quad  \left. \frac{\langle 1 -\xi^2\rangle_j(y_l^{\sigma})}{2(1-4\beta_1)\left[M(x_{j-\frac12},y_l^{\sigma})+M(x_{j+\frac12},y_l^{\sigma})\right]}\right\}.
\end{align*}
In an entirely similar manner, we obtain that for each quadrature point $x\in S_j^x$,
 $H_2(x)> 0$ if $g_h^n(x, y) \geq 0 $ on the test set $\hat S_l^y$ and
$
\mu \leq \mu_0^y
$
with
\begin{align*}
\mu_0^y=
\min_{jl}\min_{1\leq \sigma \leq L}\left\{
\frac{\langle\pm\gamma^y\mp\eta(1\pm\gamma^y)+\eta^2\rangle_l(x_j^\sigma)}{2(1\pm\gamma^y)\left(\alpha_3(\mp\gamma^y)M(x_j^{\sigma},y_{l-\frac12}) + \alpha_1(\pm\gamma^y)M(x_j^{\sigma},y_{l+\frac12})\right)}\right., \\
\quad \left. \frac{\langle 1 -\eta^2\rangle_l(x_j^{\sigma})}{2(1-4\beta_1)\left[M(x_j^{\sigma},y_{l-\frac12})+M(x_j^{\sigma}, y_{l+\frac12})\right]}\right\}.
\end{align*}
The proof is thus complete if we take
\begin{equation}\label{lambda02D}
\mu_0=\min\{\mu_0^x, \mu_0^y\}.
\end{equation}
\end{proof}
\subsection{Limiter}
To enforce the condition in Theorem \ref{th2dk2}, we can use the following scaling limiter similar to the 1D case. For all $j$ and $l$, assuming the cell averages $\bar w_{jl}>0 $.
We use the modified polynomial $\tilde w_{h}$ instead of $g_{ih}(x, y)$ on $I_{jl}$,
\begin{equation*}
\tilde{w}_{h}(x, y) = \theta \left(w_h(x, y)-\bar{w}_{jl}\right) + \bar{w}_{jl}, \text{ where }
\theta = \min\left\{1,  \frac{\bar{w}_{jl}}{\bar{w}_{jl} -\min_{ S_{jl}} w_h(x, y) } \right\}.
\end{equation*}
Similar implementation algorithm as in Section \ref{dgalg} applies in the 2D setting. 

\section{Numerical Examples}
In this section, we present  a selected set of examples to numerically validate our positivity-preserving DDG schemes.  

In the one dimensional case, for numerical approximation $u_h$ we quantify $l_1$ errors by
$$
\|u_h-u_{ref}\|_{l_1}= \sum_{j=1}^N \int_{I_j}|u_h(x) - u_{ref}(x)|dx,
$$
where each integral is evaluated by a $4$-point Gaussian quadrature method. Here $u_{ref}$ is either the exact solution in Example 1 or the fine-meshed reference solution in Example 2. 
Long time simulation is also performed to illustrate how the positivity of cell averages propagates. 

{\bf \noindent Example 1.} We consider a modified PNP so that an exact solution is available. This will help us to test numerical convergence and solution positivity.
 \label{ex1}
In $\Omega=[0,1]$, we consider
\begin{align*}
  \partial_t c_i & = \partial_x (\partial_x c_i+  q_i c_i \partial_x \psi) +f_i,\quad i=1,2 \\
- \partial_x^2 \psi & =  q_1c_1+ q_2c_2,\\
 c_1 & =x^2(1-x)^2, \quad c_2  = x^2(1-x)^3, \\
{\partial_x c_i} &+q_ic_i {\partial_x \psi}=0,  \quad x=0, 1,\\
  \psi(t,0)& =0,  \quad \partial_x \psi(t, 1)=-e^{-t}/60, 
\end{align*}
with
\begin{align*}
 f_1 &=\frac{(50x^9-198x^8+292x^7-189x^6+45x^5)}{30e^{2t}}+\frac{ (-x^4+2x^3-13x^2+12x-2)}{e^{t}},   \\
 f_2&=\frac{(x - 1)(110x^9 - 430x^8 + 623x^7 - 393x^6+90x^5)}{60e^{2t}} + \frac{(x-1)(x^4 - 2x^3 + 21x^2 - 16x + 2)}{e^{t}}.
\end{align*}
This system, with $q_1=1$ and $q_2=-1$,  admits exact solutions
\begin{align*}
  c_1 & =x^2(1-x)^2e^{-t}, \\
 c_2  &= x^2(1-x)^3e^{-t}, \\
\psi & =  -(10x^7-28x^6+21x^5)e^{-t}/420.
\end{align*}
Table \ref{tab:ex1}--\ref{tab:ex1LTsbeta1} display both the $l_1$ errors and orders of convergence at $T=0.01$ and $0.1$ with different pairs of parameters $(\beta_0,\beta_1)$. We observe that the order of convergence is roughly of $3$ in all cases
when using $P^{2}$ elements. Note that though the pair $(\beta_0, \beta_1)=(4,1/24)$ lies outside the range in Theorem \ref{thk2_1D}, we still observe the optimal order of accuracy in Table 2 and 4.


 Figure \ref{fig:ex1} shows the numerical solutions at different times. We observe that the numerical solutions (dots) match the exact solutions (solid line) at $t=0.1$ (top) and $t=1$ (bottom) very well. Our  simulation also indicates that cell averages of our numerical solutions are positive at least until $T=15$, at which time the maximum  cell average is extremely small at the level of $10^{-8}$. This simulation shows that our scheme preserves positivity in long time.


\begin{table}[!htb]
\caption{Error table of Example 1 at $t=0.01$ for $k=2$, $\beta_0=4$ and $\beta_1=\frac{1}{6}$}
\begin{tabular}{ |c|c|c|c|c|c|c| }
\hline
  h & $c_1$ error & order & $c_2$ error & order & $\psi$ error & order \\ \hline
  0.2     &    2.0624e-04  &          --   &  1.3872e-04 &        --   &8.1931e-05   &       --            \\ \hline
  0.1     &    5.4529e-05   &    2.43 &  3.5665e-05 &   2.43 &8.7833e-06  &  3.17     \\ \hline
  0.05   &   9.1862e-06    &   2.69  & 6.1522e-06  &  2.66  &  9.6769e-07 &    3.15   \\ \hline
  0.025 &  1.3082e-06     &  2.81   &   8.9366e-07 &   2.78 & 1.1148e-07 &   3.12   \\ \hline
\end{tabular}
\label{tab:ex1}
\end{table}

\begin{table}[!htb]
\caption{Error table of Example 1 at $t=0.01$ for $k=2$, $\beta_0=4$ and $\beta_1=\frac{1}{24}$}
\begin{tabular}{ |c|c|c|c|c|c|c| }
\hline
  h & $c_1$ error & order & $c_2$ error & order & $\psi$ error & order \\ \hline
  0.2     &   1.0164e-04 &          --  & 8.4562e-05  &        --     & 7.1174e-05  &         --       \\ \hline
  0.1     &   8.4066e-06 &    3.41& 7.8862e-06  &   3.44&  7.4710e-06    &  3.18  \\ \hline
  0.05   &   7.8352e-07 &    3.31& 6.7092e-07  &   3.45& 8.2247e-07   &    3.15   \\ \hline
  0.025 &   8.5408e-08 &    3.20& 6.5765e-08  &   3.35& 9.5078e-08   &   3.11  \\ \hline
\end{tabular}
\label{tab:ex1sbeta1}
\end{table}

\begin{table}[!htb]
\caption{Error table of Example 1 at $t=0.1$ for $k=2$, $\beta_0=4$ and $\beta_1=\frac{1}{6}$ }
\begin{tabular}{ |c|c|c|c|c|c|c| }
\hline
  h & $c_1$ error & order & $c_2$ error & order & $\psi$ error & order \\ \hline
 0.2  &       2.6995e-04  &         --    & 1.7356e-04  &          --   &7.5599e-05    &       -- \\ \hline
 0.1  &      6.4304e-05    &  2.51  &3.9933e-05   &   2.50  & 8.3213e-06  &   3.15  \\ \hline
 0.05  &     1.0421e-05   &   2.73 & 6.6252e-06  &    2.70 &  9.3043e-07 &    3.14    \\ \hline
 0.025 &    1.4658e-06   &   2.83 & 9.4952e-07  &    2.80 &  1.0777e-07 &      3.11   \\ \hline
\end{tabular}
\label{tab:ex1LT}
\end{table}

\begin{table}[!htb]
\caption{Error table of Example 1 at $t=0.1$ for $k=2$, $\beta_0=4$ and $\beta_1=\frac{1}{24}$ }
\begin{tabular}{ |c|c|c|c|c|c|c| }
\hline
  h & $c_1$ error & order & $c_2$ error & order & $\psi$ error & order \\ \hline
 0.2  &     9.3406e-05 &         --   &   8.1835e-05&           --  & 6.3855e-05 &          -- \\ \hline
 0.1  &      7.7940e-06  &   3.31 & 7.5466e-06  &    3.42 & 6.7668e-06 &    3.18  \\ \hline
 0.05  &   7.4802e-07 &    3.18&  6.6124e-07 &     3.41&  7.4491e-07&     3.15  \\ \hline
 0.025 &   9.4980e-08  &   2.98 &  6.7140e-08   &   3.30  & 8.5650e-08   &  3.12
    \\ \hline
\end{tabular}
\label{tab:ex1LTsbeta1}
\end{table}

\begin{figure}[!htb]
\caption{Numerical solutions versus exact solution at $t=0.1$ and $t=1$} \label{fig:ex1}
\centering
\begin{tabular}{cc}

\includegraphics[width=\textwidth,height=0.2\textheight]{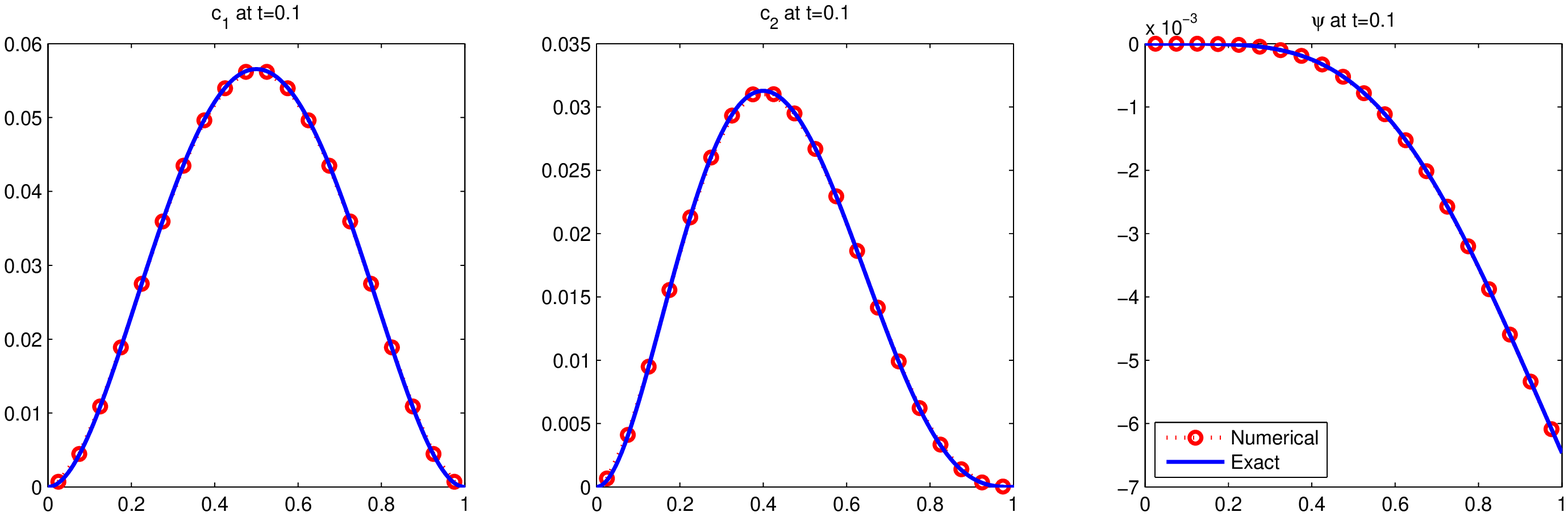}\\
\includegraphics[width=\textwidth,height=0.2\textheight]{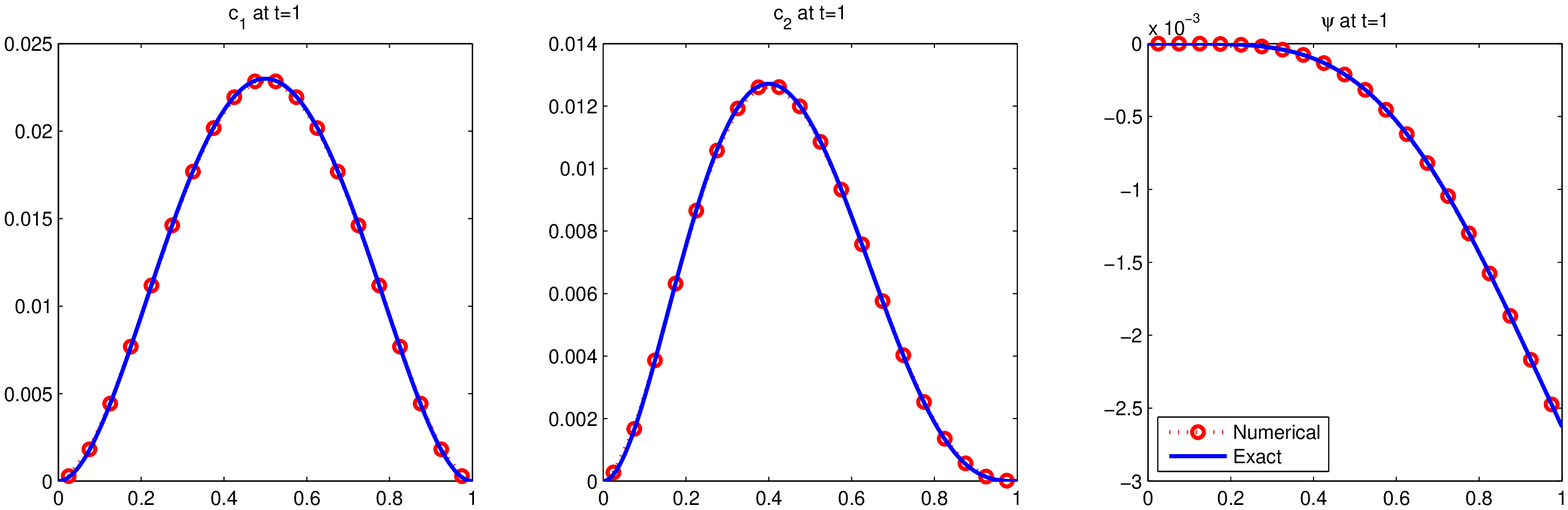} 
\end{tabular}
{Solid line: exact solutions; circles: numerical solutions}

\end{figure}.

{\bf \noindent Example 2.} In this example we solve the PNP system \eqref{PNP} with $m=2$ to show convergence, mass conservation and free energy dissipation of numerical solutions.  Defined in the domain $[0,1]$, the problem is given by 
\begin{align*}
  & \partial_t c_i =\partial_x (\partial_x c_i+ q_ic_i\partial_x \psi), \quad i=1, 2 \\ 
& - \partial_x^2\psi= q_1c_1+q_2c_2, \\ 
    & c_1^{\rm in}(x)=1+\pi\sin(\pi x), \quad c_2^{\rm in}(x)= 4-2x,\\
   & \frac{\partial c_i}{\partial  \textbf{n}}+q_ic_i \frac{\partial \psi}{\partial  \textbf{n}} =0 ,\quad x=0,1,\\
  &  {\partial_x \psi}(t,0)=0,\quad {\partial_x \psi}(t,1)=0,
\end{align*}
where $q_1$ and $q_2$ are set to be $1$ and $-1$, respectively.

With parameter pair $(\beta_0,\beta_1)=(4,1/6)$, which is admissible according to Theorem \ref{thk2_1D}, we observe the third order of accuracy in Table \ref{tab:ex2} and \ref{tab:ex2LT} at time $t=0.01$ and $t=0.1$, respectively. 

In Figure \ref{fig:ex2} (top), we see the snapshots of $c_1$, $c_2$ and $\psi$ at $t=0, 0.01, 0.1,0.8, 1$. We observe that the solutions at $t=0.8$ and $t=1$ are indistinguishable. The solution appears to be approaching to the steady state, $c_1=3$, $c_2=3$ and $\psi=0$. Figure \ref{fig:ex2} (bottom) shows the energy decay (see the change on the right vertical axis)  and conservation of mass (see the left vertical axis). We see that the total mass of $c_1$ and $c_2$ stays constant while the free energy is decreasing in time. In fact the free energy levels off after $t=0.2$, at which the system is already almost in steady state. 

\begin{table}[!htb]
\caption{Error table of Example 2 at $t=0.01$ for $k=2$, $\beta_0=4$ and $\beta_1=\frac{1}{6}$ }
\begin{tabular}{ |c|c|c|c|c|c|c| }
\hline
  h & $c_1$ error & order & $c_2$ error & order & $\psi$ error & order \\ \hline
 0.2  &      1.4182e-02 &        -- &   2.4784e-03  &        --&   1.5565e-02  &      --     \\ \hline
 0.1  &    1.8817e-03 &    2.96 & 3.3395e-04  &    2.94&  2.0471e-03  &   2.97      \\ \hline
 0.05 &    2.3522e-04 &    3.00 & 4.2034e-05  &      2.99& 2.5493e-04  &  3.01     \\ \hline
\end{tabular}
\label{tab:ex2}
\end{table}

\begin{table}[!htb]
\caption{Error table of Example 2 at $t=0.1$ for $k=2$, $\beta_0=4$ and $\beta_1=\frac{1}{6}$ }
\begin{tabular}{ |c|c|c|c|c|c|c| }
\hline
  h & $c_1$ error & order & $c_2$ error & order & $\psi$ error & order \\ \hline
 0.2  &    5.3188e-04 &         --&  5.5476e-04&         --&  5.9369e-04&         -   \\ \hline
 0.1  &    5.2491e-05 &       3.29& 7.0056e-05&     2.99& 5.8189e-05&     3.30   \\ \hline
 0.05 &    5.5591e-06 &     3.24& 8.7443e-06&     3.00& 6.1212e-06&     3.25  \\ \hline
\end{tabular}
\label{tab:ex2LT}
\end{table}

\begin{figure}[!htb]
\caption{Temporal evolution of the solutions}
\centering
\begin{tabular}{cc}
\includegraphics[width=1.05\textwidth]{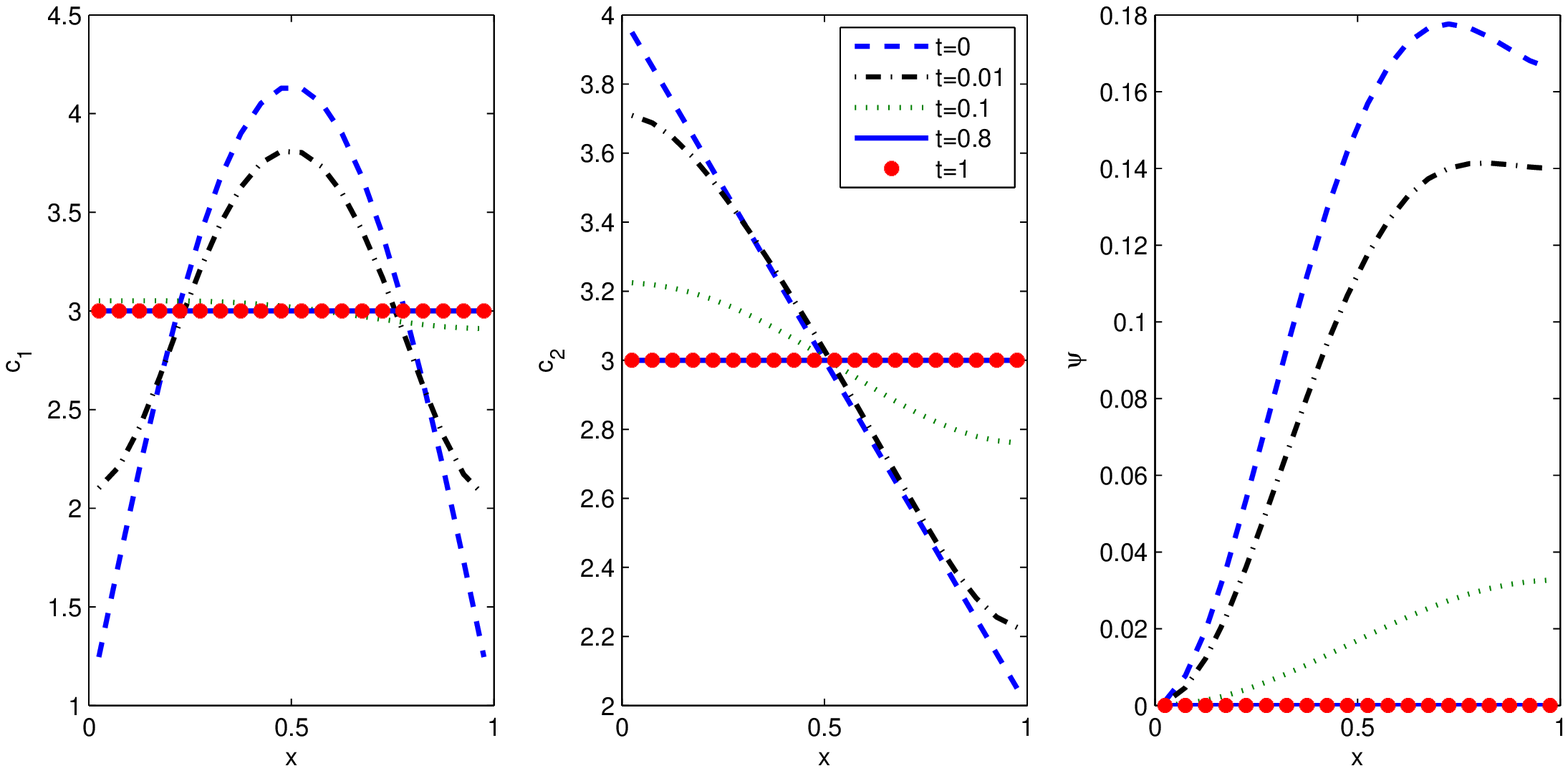}\\
\includegraphics[width=\textwidth]{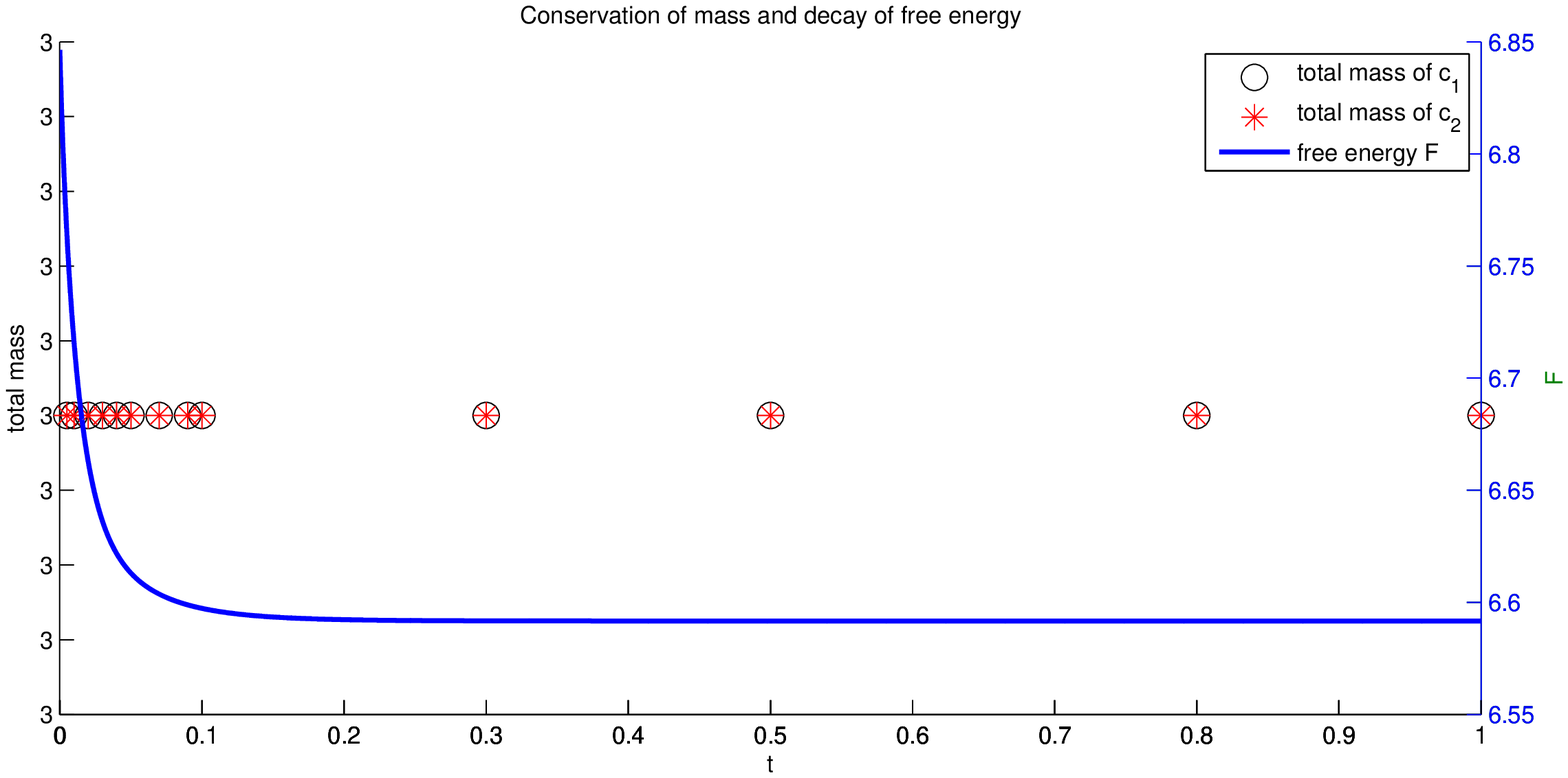} 
\end{tabular}
\label{fig:ex2}
\end{figure}


\noindent{\bf  Example 3.} This example is to test the spatial accuracy of our scheme in a 2D setting. Similar to the Numerical Test 5.1 in \cite{DWZ20}, we consider the PNP problem \eqref{PNP} on $\Omega=[0,\pi]^2$ with source terms, i.e., 
\begin{align*}
&\partial_t c_1= \nabla\cdot(\nabla c_1+c_1\nabla\psi)+f_1,\\
&\partial_t c_2= \nabla\cdot(\nabla c_2-c_2\nabla\psi)+f_2,\\
&-\Delta \psi =  c_1-c_2 + f_3.
\end{align*}
where the functions $f_i(t,x,y)$ are determined by the following exact solution
\begin{align*}
& c_1(t,x,y) = \alpha_1 \left( e^{-\alpha t} \cos(x)\cos(y) + 1\right), \\
& c_2(t,x,y)= \alpha_2 \left( e^{-\alpha t} \cos(x)\cos(y) + 1\right), \\
& \psi(t,x,y) = \alpha_3 e^{-\alpha t} \cos(x)\cos(y),
\end{align*}
where the parameters $\alpha, \alpha_1, \alpha_2$ and $\alpha_3$ will be specified in each test case.
The initial conditions in (\ref{PNP}c) are obtained by evaluating the exact solution at $t=0$, and the boundary conditions in (\ref{PNP}c) satisfy the zero flux boundary conditions. The boundary data in (\ref{PNP}d) is obtained by evaluating the exact solution $\psi(t,x,y)$ on $\partial \Omega_D$ and its normal derivative $\frac{\partial \psi}{\partial  \textbf{n}}$ on $\partial \Omega_N$. 
In each test case, we consider one of the following two different types of boundary conditions: (i) $\partial \Omega_D = \partial \Omega$, namely, $\psi$ is subject to Dirichlet boundary conditions on $ \partial \Omega$; (ii) $\partial \Omega_D=\{(x,y)\in \bar{\Omega}: x=0, x=\pi \}$ and $\partial \Omega_N=\partial \Omega \backslash \partial \Omega_D$.

The problem is solved by the fully discrete DG scheme (\ref{fully2D}) together with (\ref{dg+}bc),  modified with source  terms. The numerical flux parameters are chosen as $\beta_0=16, \beta_1=\frac{1}{6}$. The mesh ratio $\frac{\Delta t}{h^2} = 1.6\times 10^{-5}$ is taken for all test cases for it is sufficient to ensure the desired spatial accuracy.\\

{\bf \noindent Test case 3-1.} We take the parameters in the exact solution as $\alpha=\alpha_1=\alpha_2=\alpha_3=10^{-3}$. In (\ref{PNP}d), we set the boundary condition type for $\psi$ as (i).
The errors and orders of convergence at $t=0.001$ and $t=0.01$ are reported in Table \ref{tab:ex3Dacc1} and Table \ref{tab:ex3Dacc2}, respectively. From these results, we conclude that the scheme is of third order in space.

\begin{table}[!htb]
\caption{$l^1$ errors and orders at $t=0.001$ with meshes $N\times N$. }
\begin{tabular}{ |c|c|c|c|c|c|c| }
\hline
 $N$ & $c_1$ error & order & $c_2$ error & order & $\psi$ error & order \\ \hline
 10  &      2.1205e-06  &         --    & 2.1205e-06  &          --   & 2.4658e-06    &       -- \\ \hline
 20  &    2.7552e-07    &  2.94  & 2.7552e-07   &   2.94  & 2.8821e-07  &   3.10  \\ \hline
 30  &    8.1374e-08   &  3.01 & 8.1374e-08  &  3.01 & 8.3538e-08 &  3.05    \\ \hline
40 &   3.4254e-08   &  3.00 & 3.4254e-08  &  3.01 &  3.4884e-08 &   3.04   \\ \hline
\end{tabular}
\label{tab:ex3Dacc1}
\end{table}

\begin{table}[!htb]
\caption{$l^1$ errors and orders at $t=0.01$ with meshes $N\times N$. }
\begin{tabular}{ |c|c|c|c|c|c|c| }
\hline
 $N$ & $c_1$ error & order & $c_2$ error & order & $\psi$ error & order \\ \hline
 10  &   2.22520e-06  &         --    & 2.22520e-06  &          --   & 2.46582e-06    &       -- \\ \hline
 20  &    2.75811e-07   &  3.01  & 2.75812e-07   &   3.01  & 2.88210e-07  &   3.10  \\ \hline
 30  &   8.13896e-08   &  3.01 & 8.13899e-08  &  3.01 & 8.35368e-08 &  3.05    \\ \hline
40 &   3.42590e-08   &  3.01 & 3.42591e-08  &   3.01 &  3.48838e-08 &   3.04   \\ \hline
\end{tabular}
\label{tab:ex3Dacc2}
\end{table}

{\bf \noindent Test case 3-2.} We still take  $\alpha=\alpha_1=\alpha_2=\alpha_3=10^{-3}$ in the exact solution and set the boundary type for $\psi$ as (ii). The errors and orders of convergence at $t=0.001$ and $t=0.01$ are reported in Table \ref{tab:ex3Dacc3} and Table \ref{tab:ex3Dacc4}, respectively. From these results, we find again that the scheme is of $3$rd order in space.

\begin{table}[!htb]
\caption{$l^1$ errors and orders at $t=0.001$ with meshes $N\times N$. }
\begin{tabular}{ |c|c|c|c|c|c|c| }
\hline
 $N$ & $c_1$ error & order & $c_2$ error & order & $\psi$ error & order \\ \hline
 10  &       2.2132e-06  &         --    & 2.2132e-06  &          --   & 2.5216e-06    &       -- \\ \hline
 20  &    2.7552e-07    &  3.01  & 2.7552e-07   &   3.01  & 2.8898e-07  &   3.13  \\ \hline
 30  &    8.1374e-08   &  3.01 & 8.1374e-08  &  3.01 & 8.3559e-08 &  3.06    \\ \hline
40 &   3.4254e-08   &  3.01 & 3.4254e-08  &  3.01 &  3.4875e-08 &   3.04   \\ \hline
\end{tabular}
\label{tab:ex3Dacc3}
\end{table}

\begin{table}[!htb]
\caption{$l^1$ errors and orders at $t=0.01$ with meshes $N\times N$. }
\begin{tabular}{ |c|c|c|c|c|c|c| }
\hline
 $N$ & $c_1$ error & order & $c_2$ error & order & $\psi$ error & order \\ \hline
  10  &       2.22520e-06  &         --    & 2.22520e-06  &          --   & 2.52159e-06    &       -- \\ \hline
 20  &    2.75811e-07    &  3.01  & 2.75812e-07   &   3.01  & 2.88977e-07  &   3.13  \\ \hline
 30  &    8.13897e-08   &  3.01 & 8.13898e-08  &  3.01 & 8.35586e-08 &  3.06    \\ \hline
40 &   3.42590e-08   &  3.01 & 3.42591e-08  &  3.01 &  3.48746e-08 &   3.04   \\ \hline
\end{tabular}
\label{tab:ex3Dacc4}
\end{table}

{\bf \noindent Test case 3-3.} In this test case, we test the spatial accuracy of our scheme for a larger $T=0.1$ with the following parameters $\alpha=\alpha_1=2\alpha_2=\alpha_3=10^{-2}$ in the exact solution and use the boundary type (ii) for $\psi$. The errors and orders of convergence at $t=0.001$, $t=0.01$ and $t=0.1$ are reported in Table \ref{tab:ex3Dacc5}, Table \ref{tab:ex3Dacc6} and Table \ref{tab:ex3Dacc7}, respectively. These results further confirm third order of accuracy in space.\\

\begin{table}[!htb]
\caption{$l^1$ errors and orders at $t=0.001$ with meshes $N\times N$. }
\begin{tabular}{ |c|c|c|c|c|c|c| }
\hline
 $N$ & $c_1$ error & order & $c_2$ error & order & $\psi$ error & order \\ \hline
 10  &      4.2411e-05  &         --    & 2.1207e-05  &          --   & 5.0437e-05    &       -- \\ \hline
 20  &    5.5113e-06    &  2.94  & 2.7558e-06   &   2.94  & 5.7797e-06  &   3.13  \\ \hline
 30  &    1.6277e-06   &  3.01 & 8.1389e-07  &  3.01 & 1.6712e-06 &  3.06    \\ \hline
40 &   6.8519e-07   &  3.01 & 3.4260e-07  &  3.01 &  6.9750e-07 &   3.04   \\ \hline
\end{tabular}
\label{tab:ex3Dacc5}
\end{table}

\begin{table}[!htb]
\caption{$l^1$ errors and orders at $t=0.01$ with meshes $N\times N$. }
\begin{tabular}{ |c|c|c|c|c|c|c| }
\hline
 $N$ & $c_1$ error & order & $c_2$ error & order & $\psi$ error & order \\ \hline
 10  &   4.45250e-05  &         --    & 2.22621e-05  &          --   & 5.04492e-05    &       -- \\ \hline
 20  &   5.51716e-06   &  3.01  & 2.75865e-06   &   3.01  & 5.77974e-06  &   3.13  \\ \hline
 30  &   1.62795e-06   &  3.01 & 8.13998e-07  &  3.01 & 1.67113e-06 &  3.06    \\ \hline
40 &   6.85243e-07   &  3.01 & 3.42630e-07  &   3.01 &  6.97456e-07 &   3.04  \\ \hline
\end{tabular}
\label{tab:ex3Dacc6}
\end{table}

\begin{table}[!htb]
\caption{$l^1$ errors and orders at $t=0.1$ with meshes $N\times N$. }
\begin{tabular}{ |c|c|c|c|c|c|c| }
\hline
 $N$ & $c_1$ error & order & $c_2$ error & order & $\psi$ error & order \\ \hline
 10  &   4.50511e-05  &         --    & 2.25252e-05  &          --   & 5.05368e-05    &       -- \\ \hline
 20  &   5.53926e-06   &  3.02  & 2.76971e-06   &   3.02  & 5.77936e-06  &   3.13  \\ \hline
 30  &   1.63136e-06   &  3.01 &  8.15701e-07  &  3.01 & 1.67036e-06 &  3.06    \\ \hline
40 &   6.86088e-07  &  3.01 & 3.43052e-07  &   3.01 &  6.97023e-07 &   3.04   \\ \hline
\end{tabular}
\label{tab:ex3Dacc7}
\end{table}

{\bf \noindent Test case 3-4.} 
In this test case, we test the spatial accuracy based on various values of $\alpha, \alpha_1, \alpha_2, \alpha_3$ at $T=0.01$ when $\psi$ is of boundary type (i). The $l^1$ errors and orders of convergence at $t=0.01$ with mesh $N\times N$ are shown in Table 
\ref{tab:ex3Dacc8}-\ref{tab:ex3Dacc10}, and the parameters $\alpha, \alpha_1, \alpha_2, \alpha_3$ are specified in each table. All these cases indicate that our scheme is stable in producing solutions of third order of accuracy in space.

\begin{table}[!htb]
\caption{$\alpha=\alpha_1=2\alpha_2=\alpha_3=1$.}
\begin{tabular}{ |c|c|c|c|c|c|c| }
\hline
 $N$ & $c_1$ error & order & $c_2$ error & order & $\psi$ error & order \\ \hline
 10  &      4.70196e-03  &         --    & 1.25563e-03  &          --   & 2.44062e-03    &       -- \\ \hline
 20  &    5.57018e-04    &  3.08  & 1.49662e-04   &   3.07  & 2.85363e-04  &   3.10  \\ \hline
 30  &    1.60997e-04   &  3.06 & 4.35957e-05  &  3.04 & 8.27135e-05 &  3.05    \\ \hline
40 &   6.70705e-05   &  3.04 & 1.82319e-05  &  3.03 &  3.45401e-05  &   3.04   \\ \hline
\end{tabular}
\label{tab:ex3Dacc8}
\end{table}

\begin{table}[!htb]
\caption{$\alpha=\alpha_1=\alpha_2=\alpha_3=1$.}
\begin{tabular}{ |c|c|c|c|c|c|c| }
\hline
 $N$ & $c_1$ error & order & $c_2$ error & order & $\psi$ error & order \\ \hline
 10  &     4.70152e-03  &         --    & 2.51073e-03  &          --   & 2.44014e-03    &       -- \\ \hline
 20  &    5.57010e-04    &  3.08  & 2.99314e-04   &   3.07  & 2.85349e-04  &   3.10  \\ \hline
 30  &    1.60996e-04   &  3.06 & 8.71903e-05  &  3.04 & 8.27119e-05 &  3.05    \\ \hline
40 &   6.70703e-05   &  3.04 & 3.64635e-05  &  3.03 & 3.45398e-05 &   3.04   \\ \hline
\end{tabular}
\label{tab:ex3Dacc9}
\end{table}

\begin{table}[!htb]
\caption{$2\alpha=\alpha_1=\alpha_2=\alpha_3=2$.}
\begin{tabular}{ |c|c|c|c|c|c|c| }
\hline
 $N$ & $c_1$ error & order & $c_2$ error & order & $\psi$ error & order \\ \hline
 10  &   2.91466e-02  &         --    & 7.76994e-03  &          --   & 4.87560e-03   &       -- \\ \hline
 20  &  3.54794e-03    &  3.04  & 9.89917e-04   &   2.97  & 5.70727e-04  &   3.09  \\ \hline
 30  &    1.02805e-03   &  3.06 & 2.92317e-04  &  3.01 & 1.65447e-04 &  3.05    \\ \hline
40 &   4.27907e-04   &  3.05 & 1.22839e-04  &  3.01 &  6.90885e-05  &   3.04   \\ \hline
\end{tabular}
\label{tab:ex3Dacc10}
\end{table}

{\bf \noindent Example 4.} (Solution positivity, conservation of mass and decay of free energy) We test our scheme for the 2D PNP system \eqref{PNP} with $m=2$ on the domain $[0,1]^2$,
\begin{align*}
&\partial_t c_1= \nabla\cdot(\nabla c_1+c_1\nabla\psi),\\
&\partial_t c_2= \nabla\cdot(\nabla c_2-c_2\nabla\psi),\\
&-\Delta \psi =  c_1-c_2,\\
& c_1^{\rm in}(x,y)=\frac{1}{20}\left(\pi\sin(\pi x)+\pi\sin(\pi y)\right), \\
& c_2^{\rm in}(x,y)=x^2(1-x)^2+y^2(1-y)^2,\\
& \frac{\partial c_i}{\partial  \textbf{n}}+q_ic_i \frac{\partial \psi}{\partial  \textbf{n}} =0 ,\quad (x,y) \in \partial \Omega,\\
&\psi =0 \mbox{~~on~} \partial\Omega_D,  \mbox{~~and~} \frac{\partial \psi}{\partial  \textbf{n}}  =0  \mbox{~~on~} \partial\Omega_N, \quad t>0,
\end{align*}
where $\partial \Omega_D=\{(x,y)\in \bar{\Omega}: x=0, x=1 \}$ and $\partial \Omega_N=\partial \Omega \backslash \partial \Omega_D$.
We solve this problem by the fully discrete DG scheme (\ref{fully2D}) with the numerical flux parameters $\beta_0=16, \beta_1=\frac{1}{6}$. \\

{\bf \noindent Test case 4-1.} (Solution positivity)
For $t\in (0,1]$, the smallest cell averages of $c_1, c_2$, and the minimum values of $g_1, g_2$ on set $S_{jl}$ based on mesh $20\times 20$ and time step $\Delta t = 10^{-5}$ are shown in Figure \ref{ex4cg}(a) and \ref{ex4cg}(b), respectively. 
From Figure \ref{ex4cg}, we find that the smallest cell averages of $c_l, c_2$ and the minimum values of $g_1, g_2$ on set $S_{jl}$ are positive when they are away from $t=0$. To have a clear view near $t=0$, we test the problem again for $t\in (0, 10^{-5}]$ based on mesh $40\times 40$ and time step $\Delta t = 10^{-7}$. The smallest cell averages of $c_1, c_2$ are now shown in Figure \ref{ex4avg}, which together with Figure \ref{ex4cg}(a) indicates that our scheme can preserve the positivity of the cell averages of $c_1, c_2$. The minimum values of $g_1, g_2$ before and after using the limiter near $t=0$ are shown in Figure \ref{ex4ming}. 
From the comparison shown in Figure \ref{ex4ming}, we find the limiter is effective in preserving the positivity of the minimum values of $g_1, g_2$ on set $S_{jl}$, hence, preserving the positivity of the cell averages of $c_1, c_2$. \\

\begin{figure}
\centering
\subfigure[]{\includegraphics[width=0.49\textwidth]{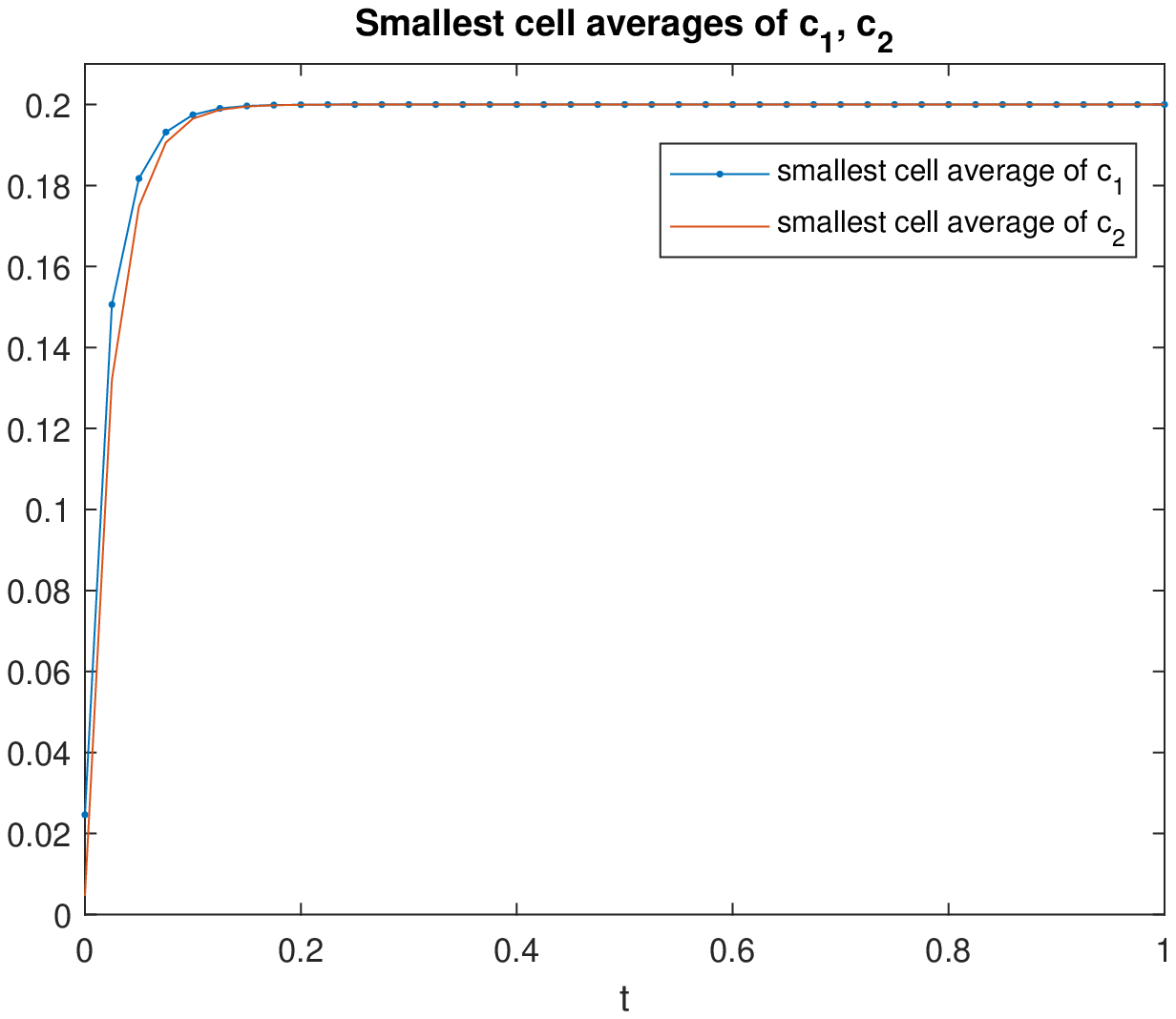}}
\subfigure[]{\includegraphics[width=0.49\textwidth]{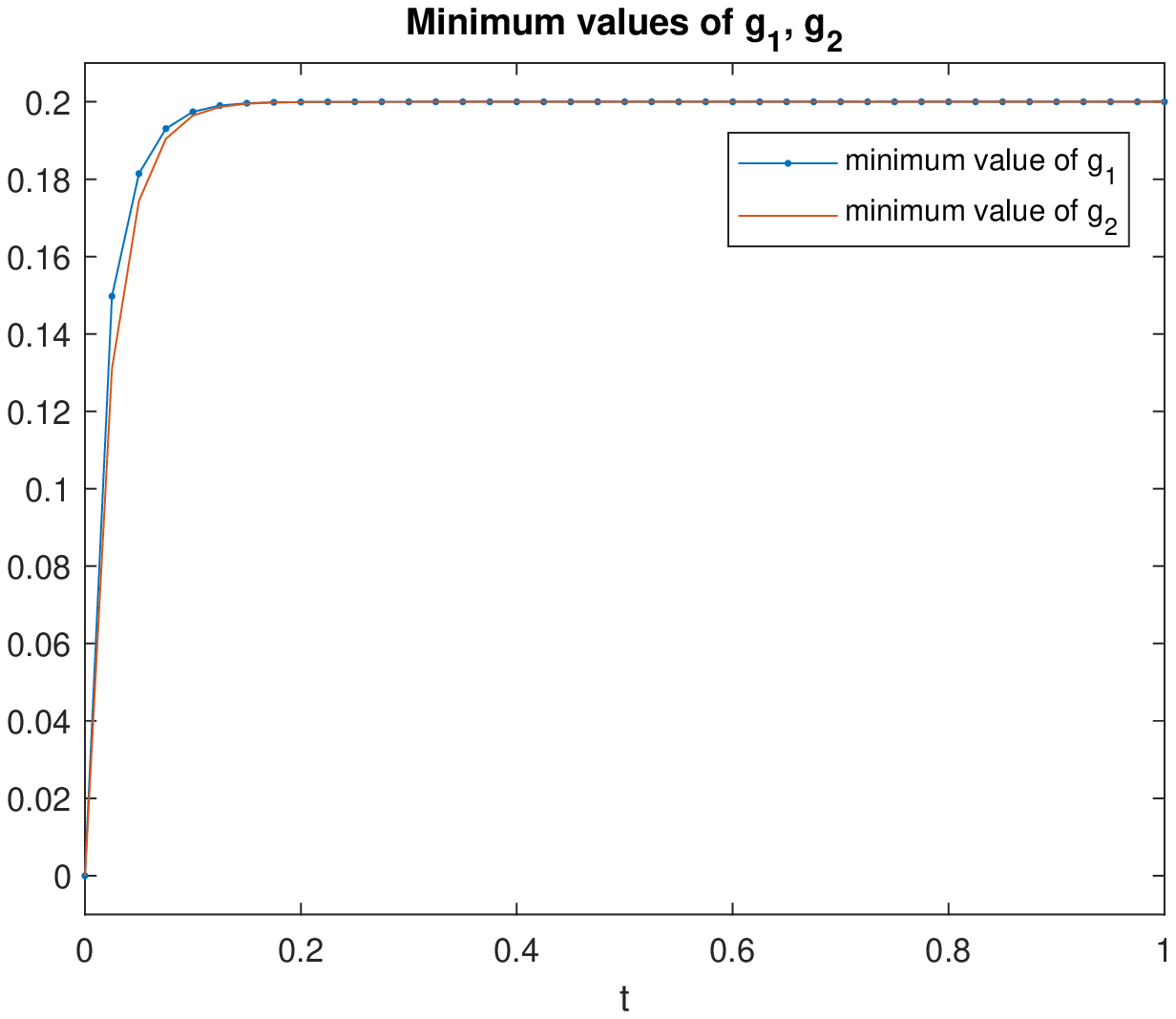}}
\caption{ Smallest cell averages of $c_1, c_2$ and the minimum values of $g_1, g_2$.
  } \label{ex4cg}
 \end{figure}

\begin{figure}
\centering
\subfigure[]{\includegraphics[width=0.49\textwidth]{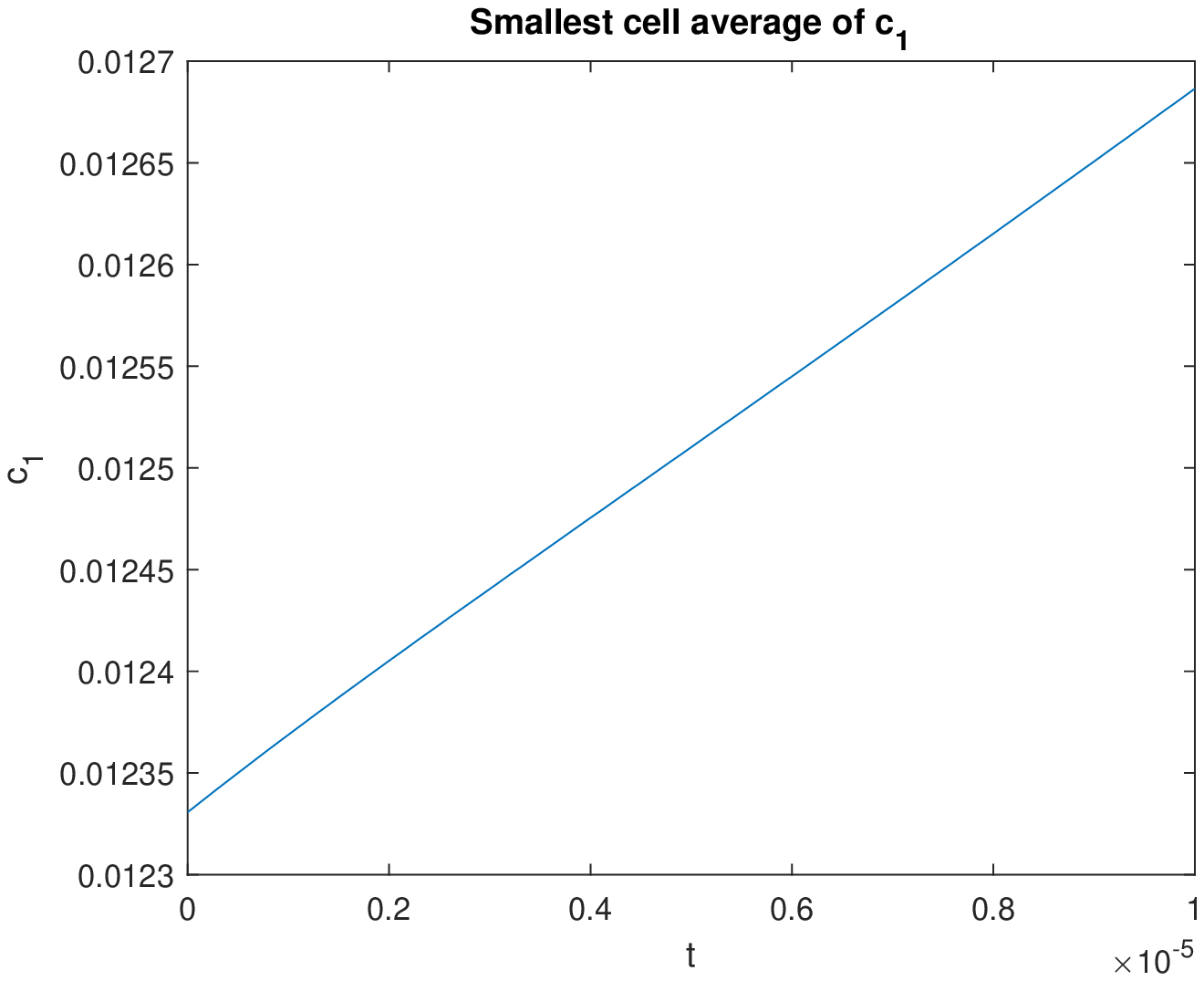}}
\subfigure[]{\includegraphics[width=0.49\textwidth]{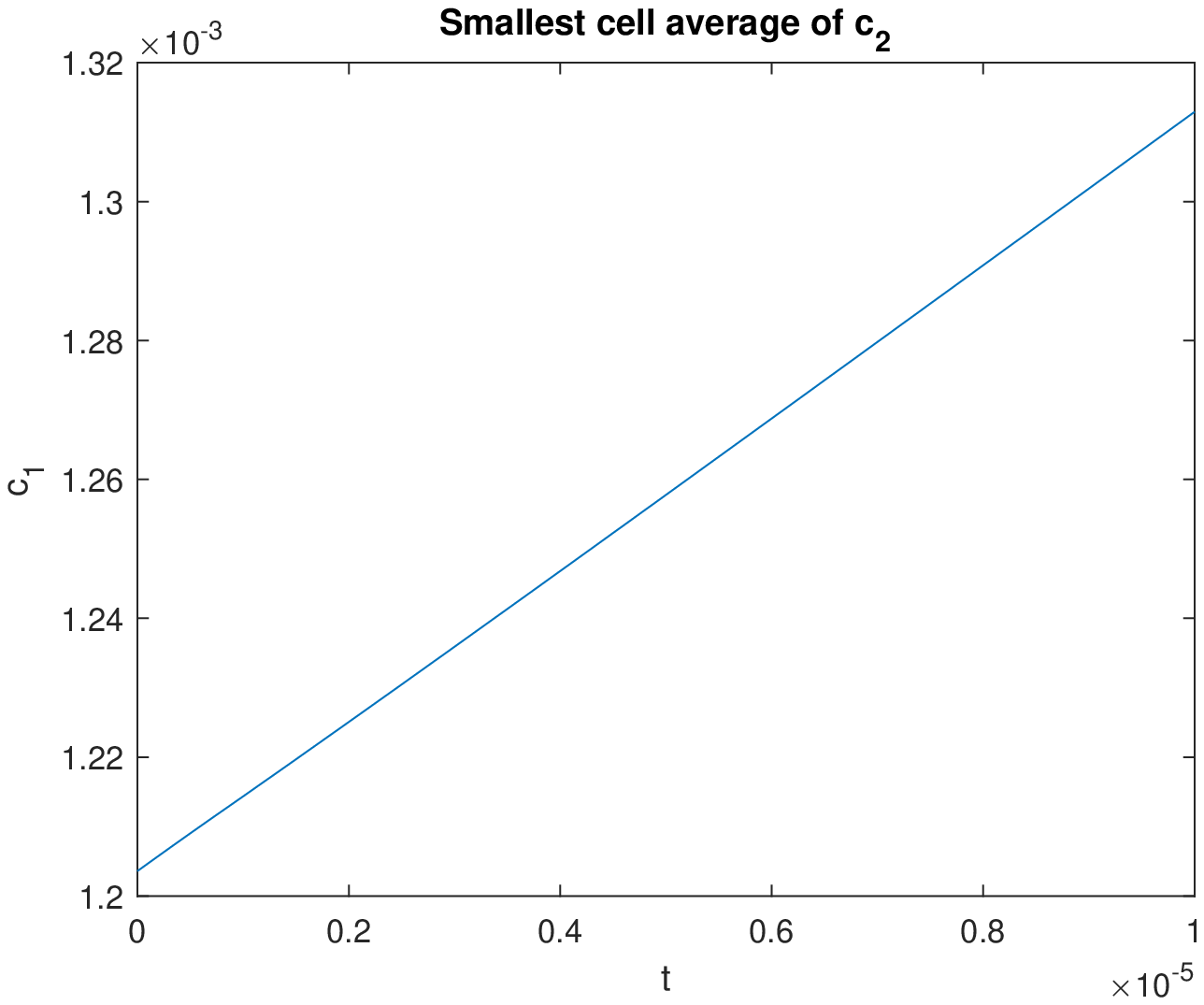}}
\caption{ Smallest cell averages of $c_1, c_2$.
  } \label{ex4avg}
 \end{figure}
 
 \begin{figure}
\centering
\subfigure[]{\includegraphics[width=0.49\textwidth]{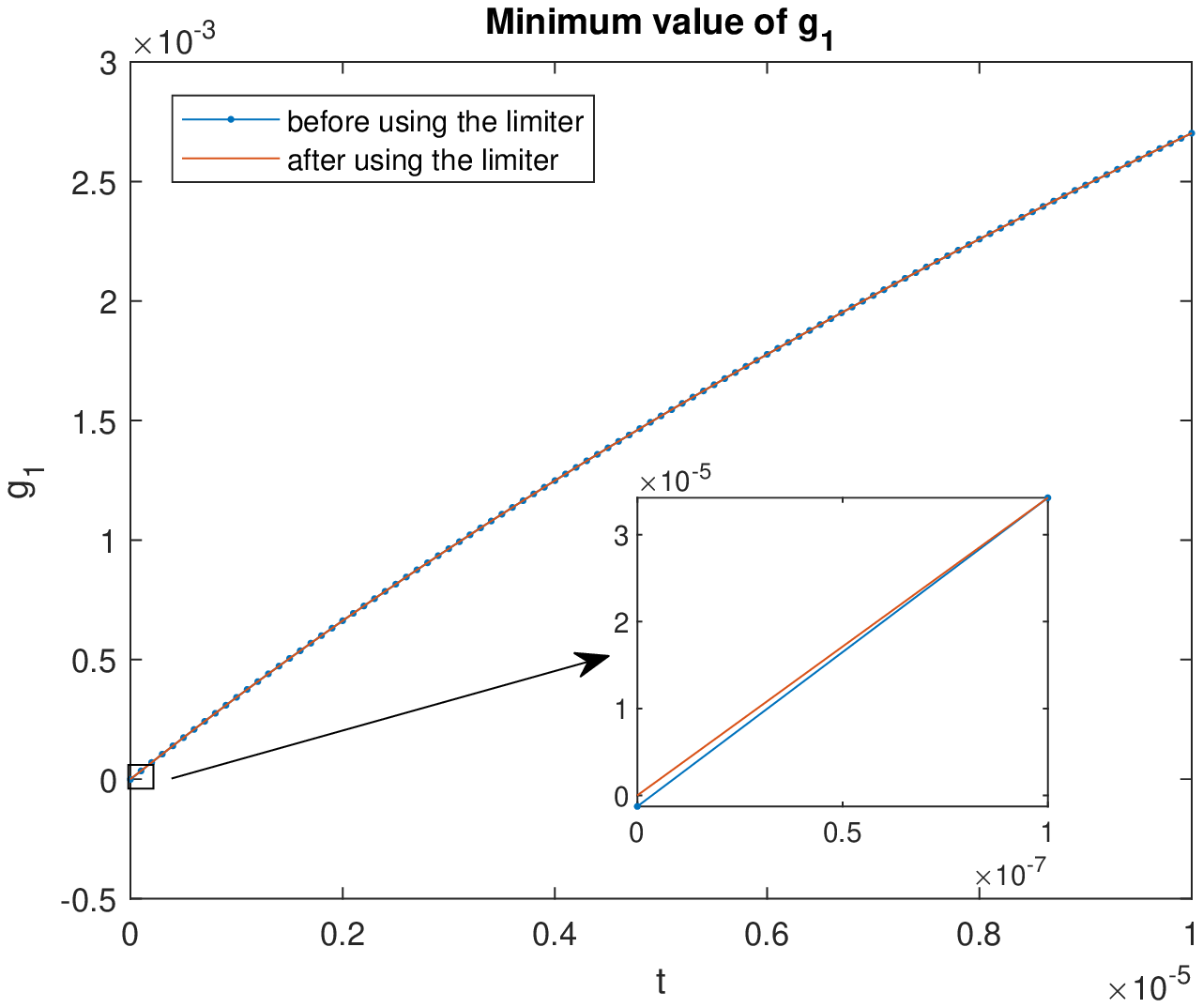}}
\subfigure[]{\includegraphics[width=0.49\textwidth]{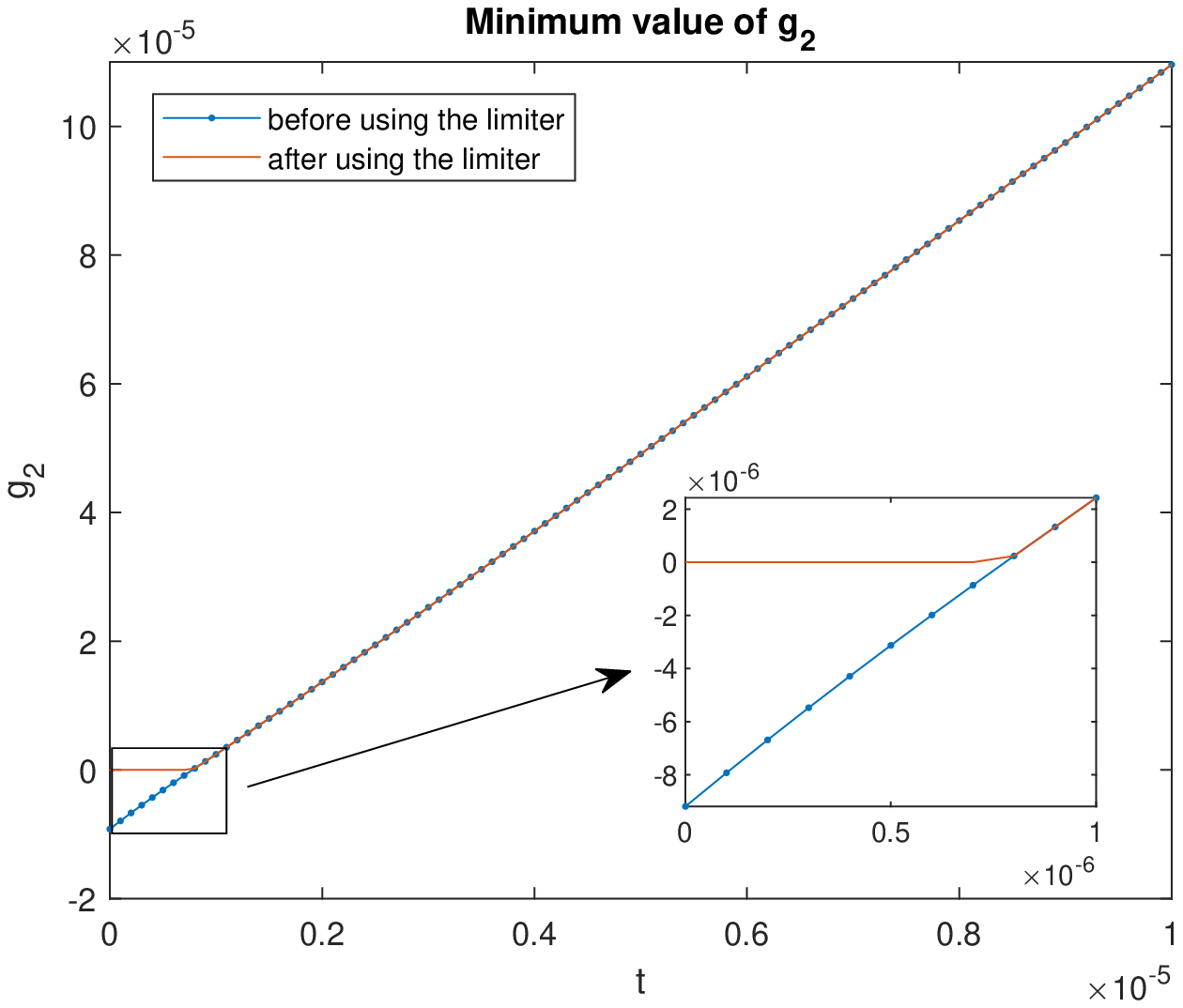}}
\caption{ Minimum values of $g_1,g_2$ on set $S_{jl}$.
  } \label{ex4ming}
 \end{figure} 

{\bf \noindent Test case 4-2.} (conservation of mass and decay of free energy)
Next, we simulate the 
evolution of $c_1, c_2$ and $\psi$ for $t\in (0,0.5]$. 
The contours of $c_1-0.2$ (first column), $c_2-0.2$ (second column) and $\psi$ (right) at $t=0, 0.01, 0.1, 0.5$ are shown in Figure \ref{ex5pattern}. We observe that the contours at $t=0.1$ and $t=0.5$ are indistinguishable. The solution appears approaching the steady state, $c_1=0.2$, $c_2=0.2$ and $\psi=0$.

Figure \ref{fig:ex5} shows the energy decay (see the change on the right vertical axis)  and conservation of mass (see the left vertical axis). Similar results are observed as in Example 2, which confirms the conservation of mass and dissipation of energy.


\begin{figure}
\centering
\subfigure{\includegraphics[width=0.325\textwidth]{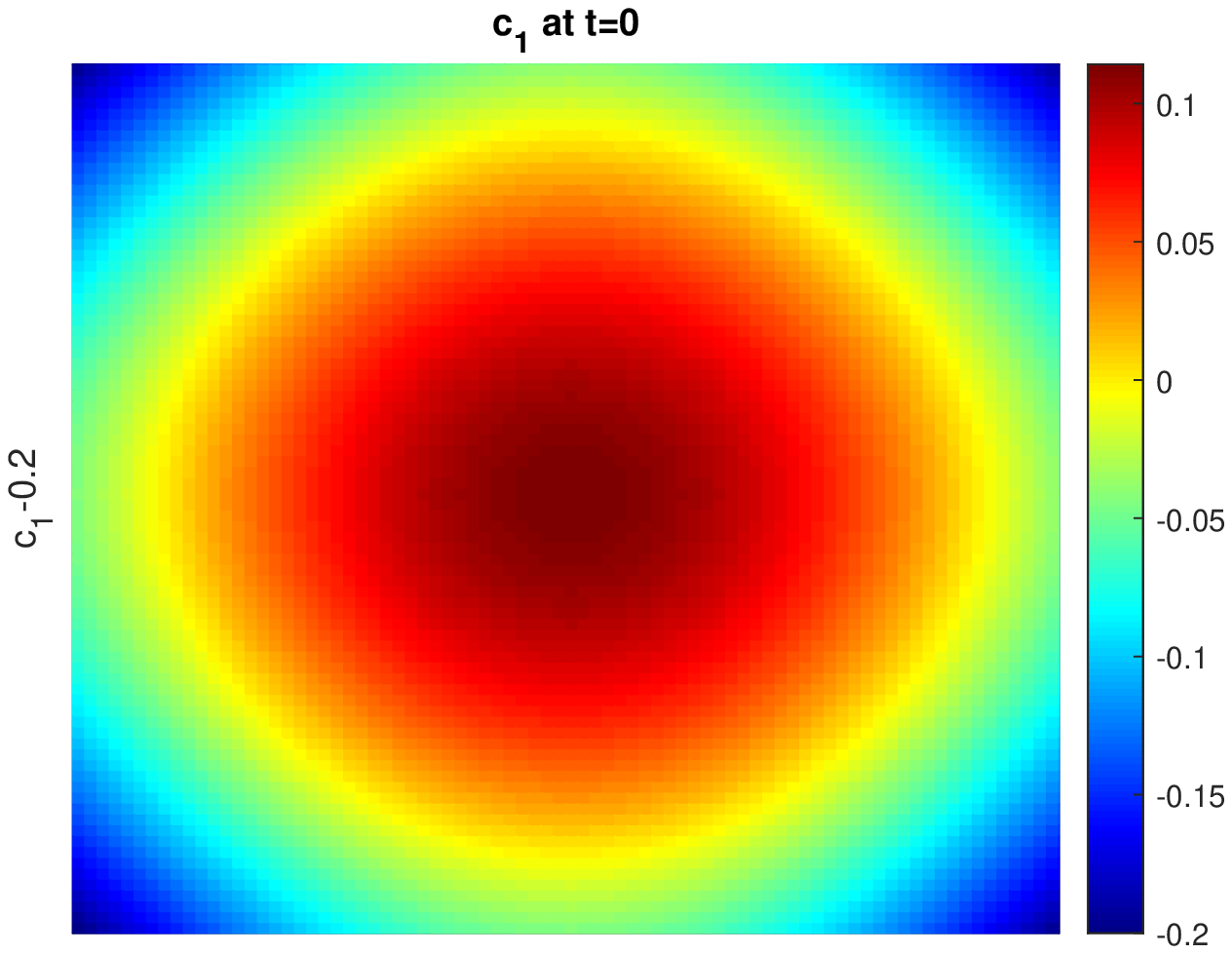}}
\subfigure{\includegraphics[width=0.325\textwidth]{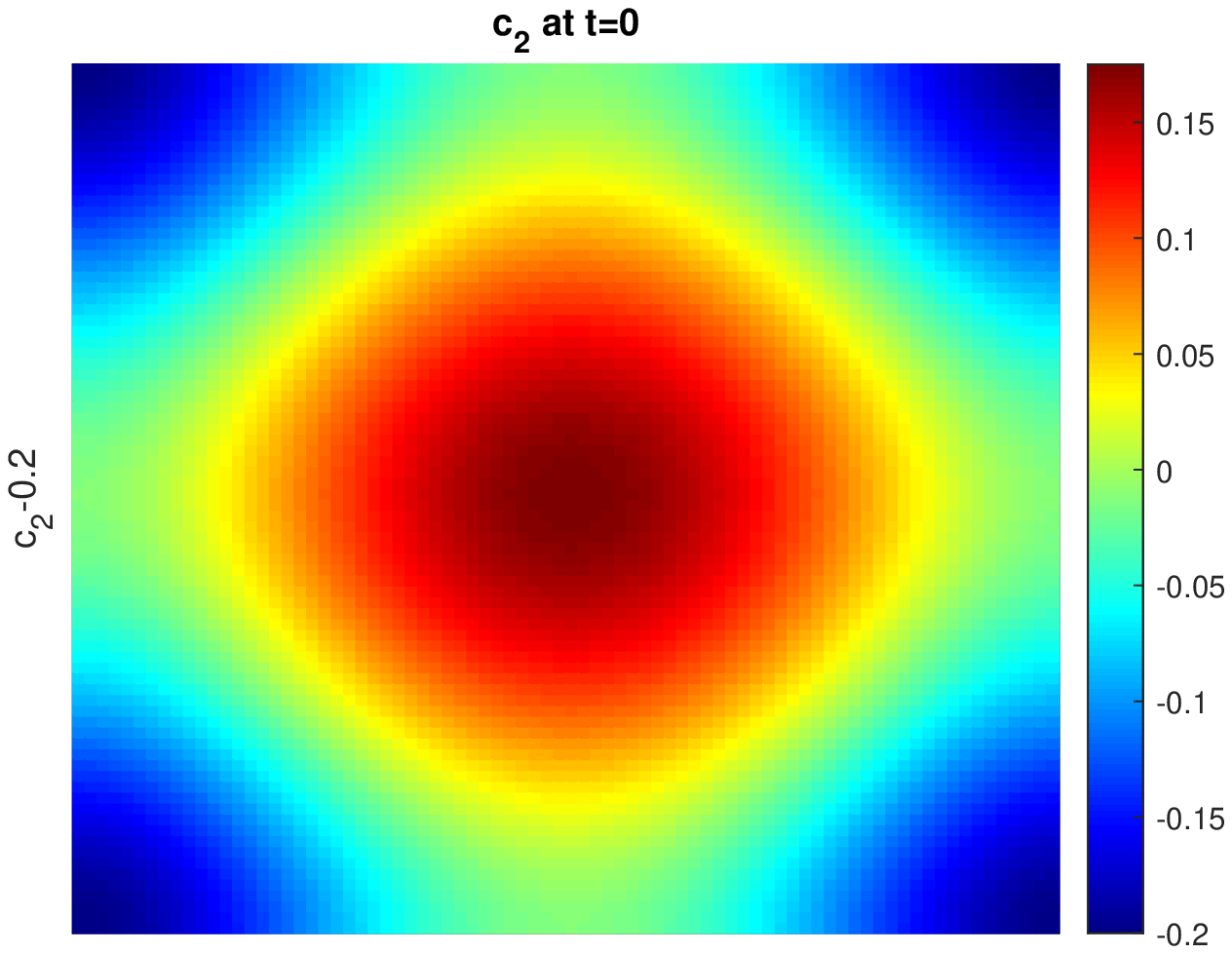}}
\subfigure{\includegraphics[width=0.325\textwidth]{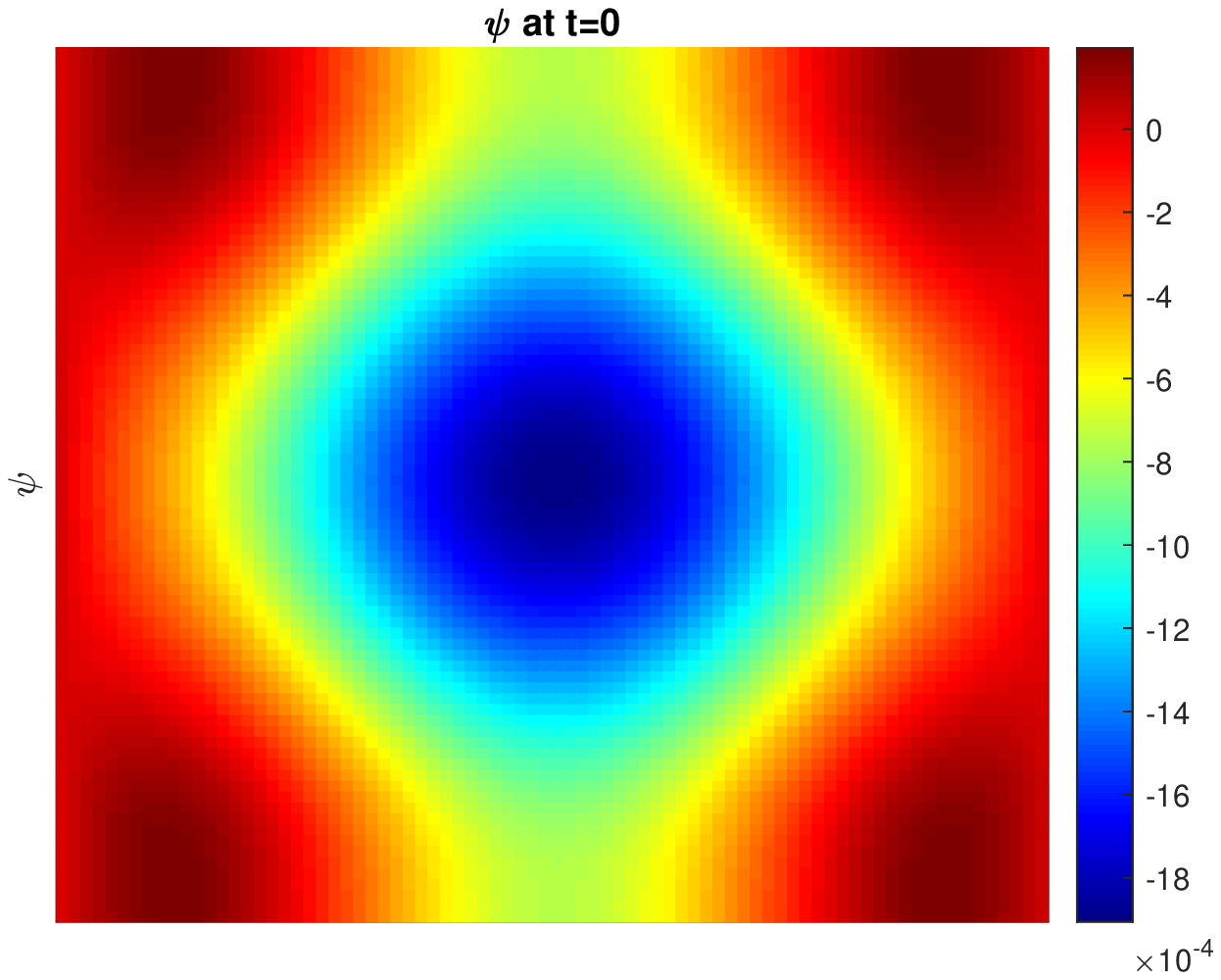}}
\subfigure{\includegraphics[width=0.325\textwidth]{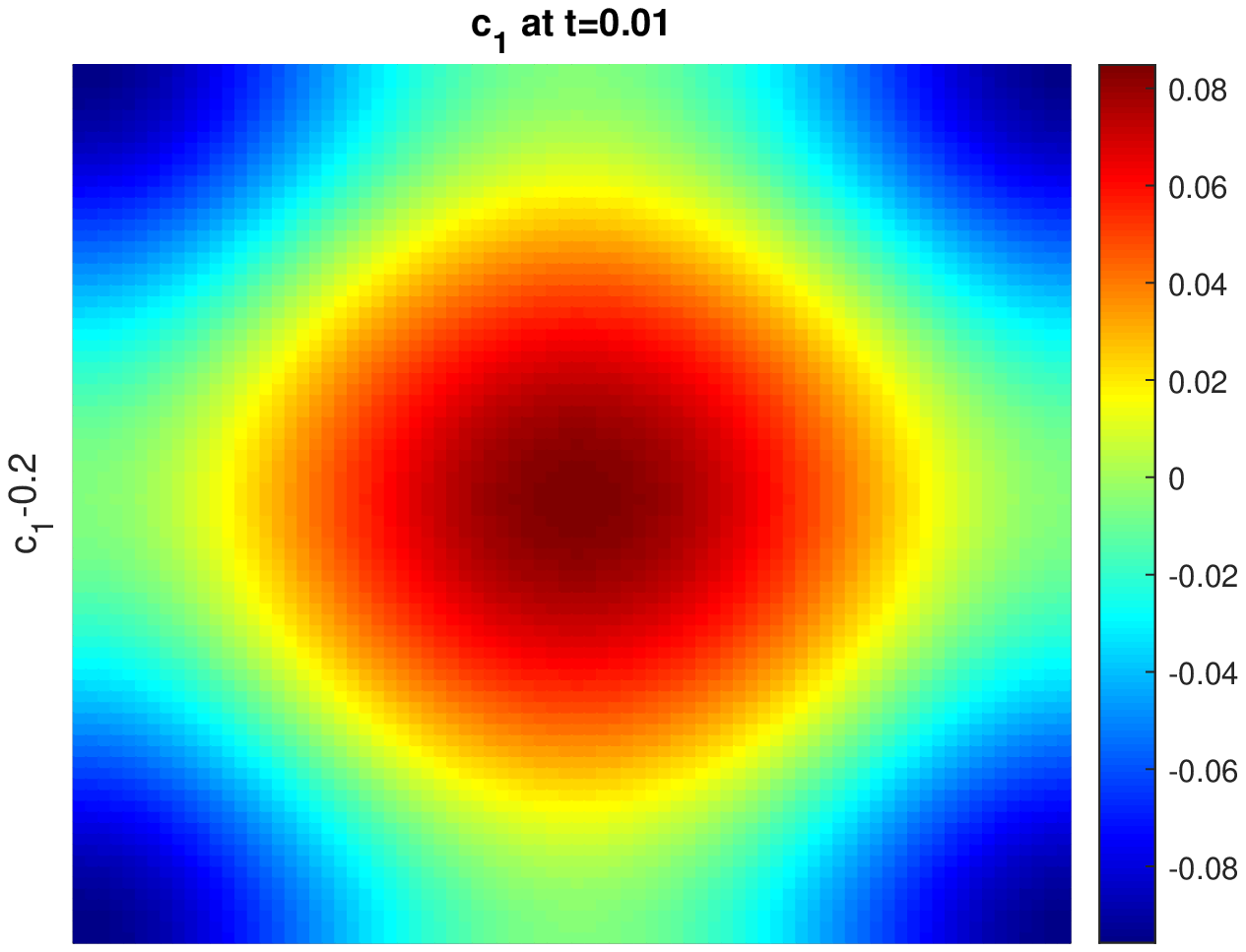}}
\subfigure{\includegraphics[width=0.325\textwidth]{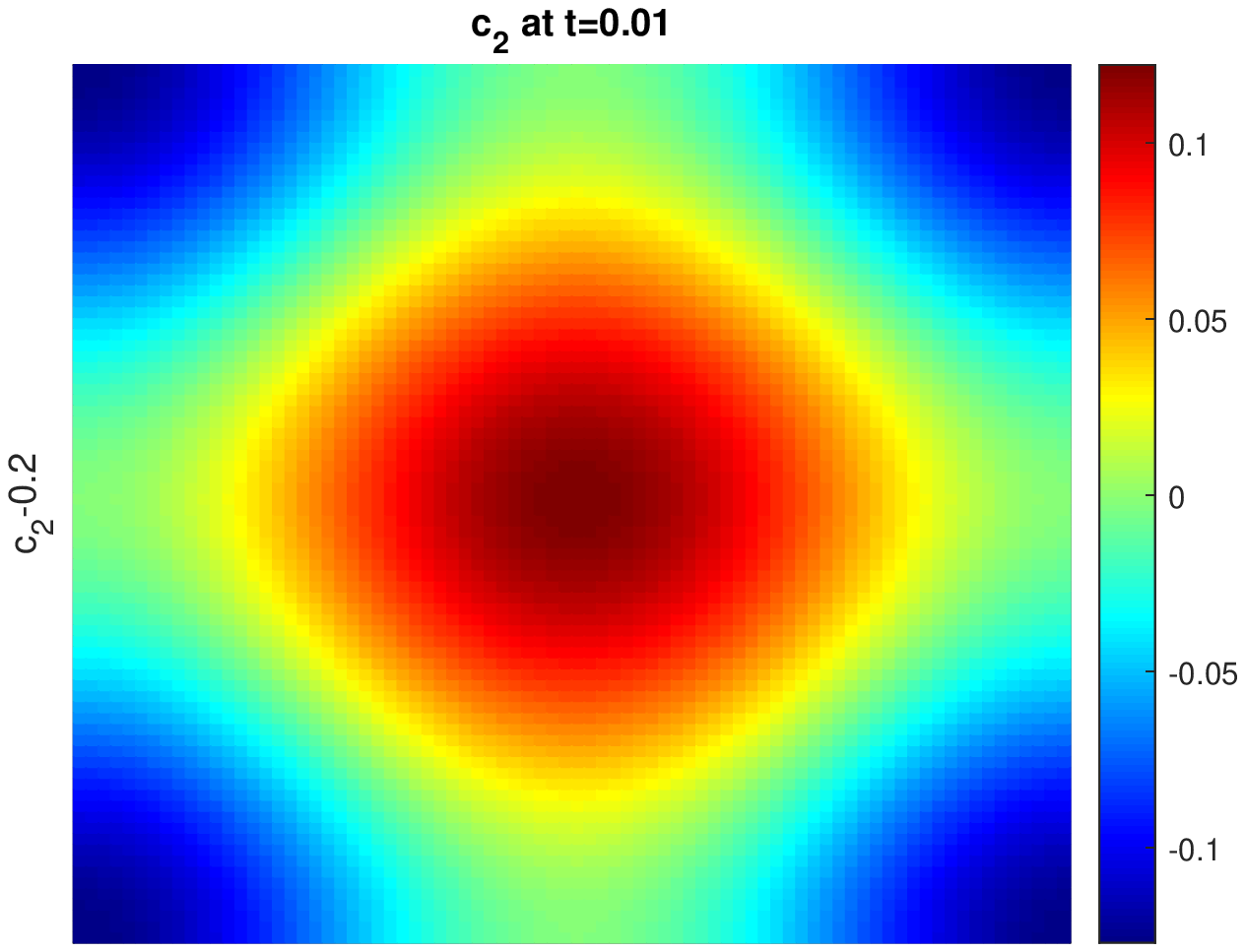}}
\subfigure{\includegraphics[width=0.325\textwidth]{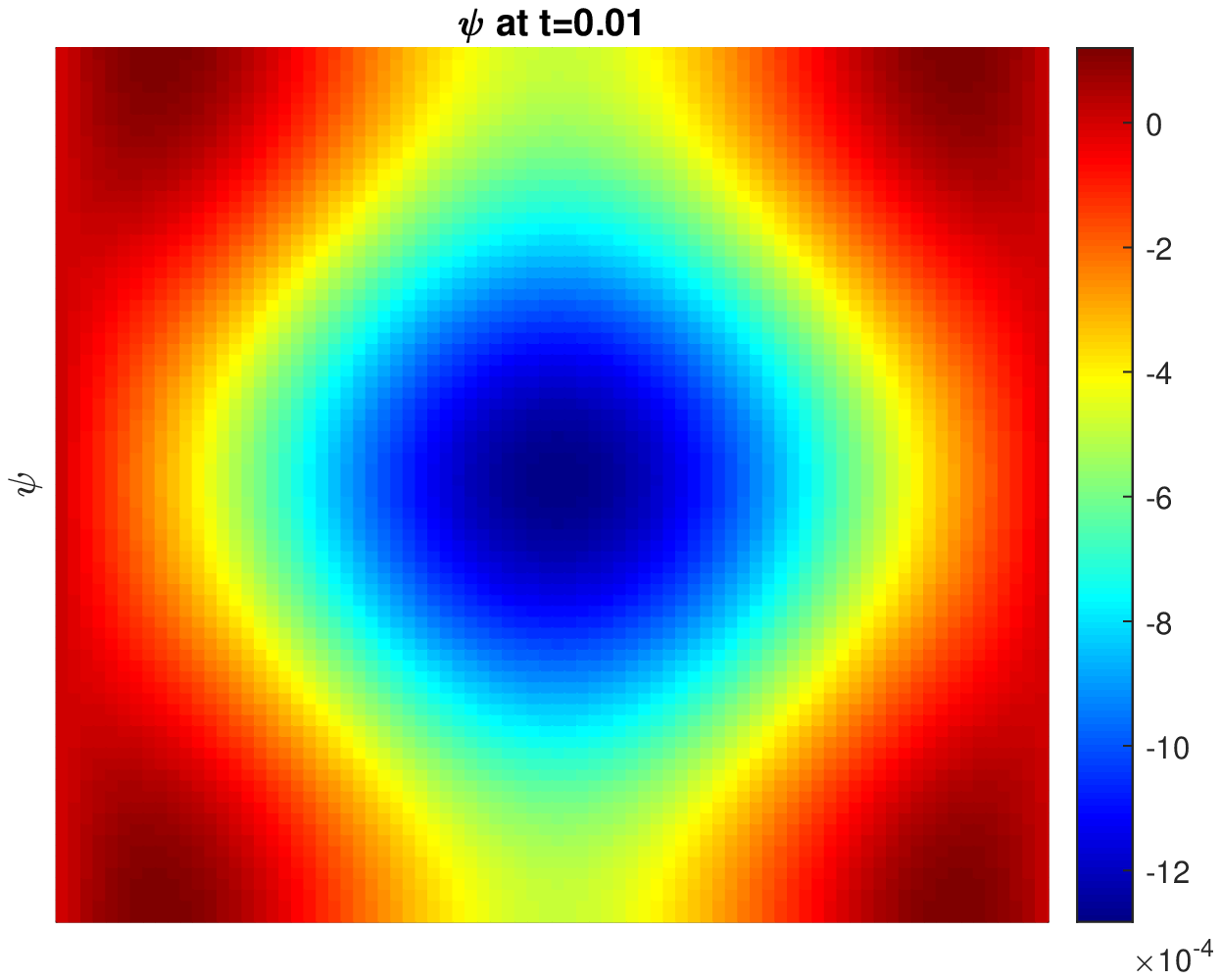}}
\subfigure{\includegraphics[width=0.325\textwidth]{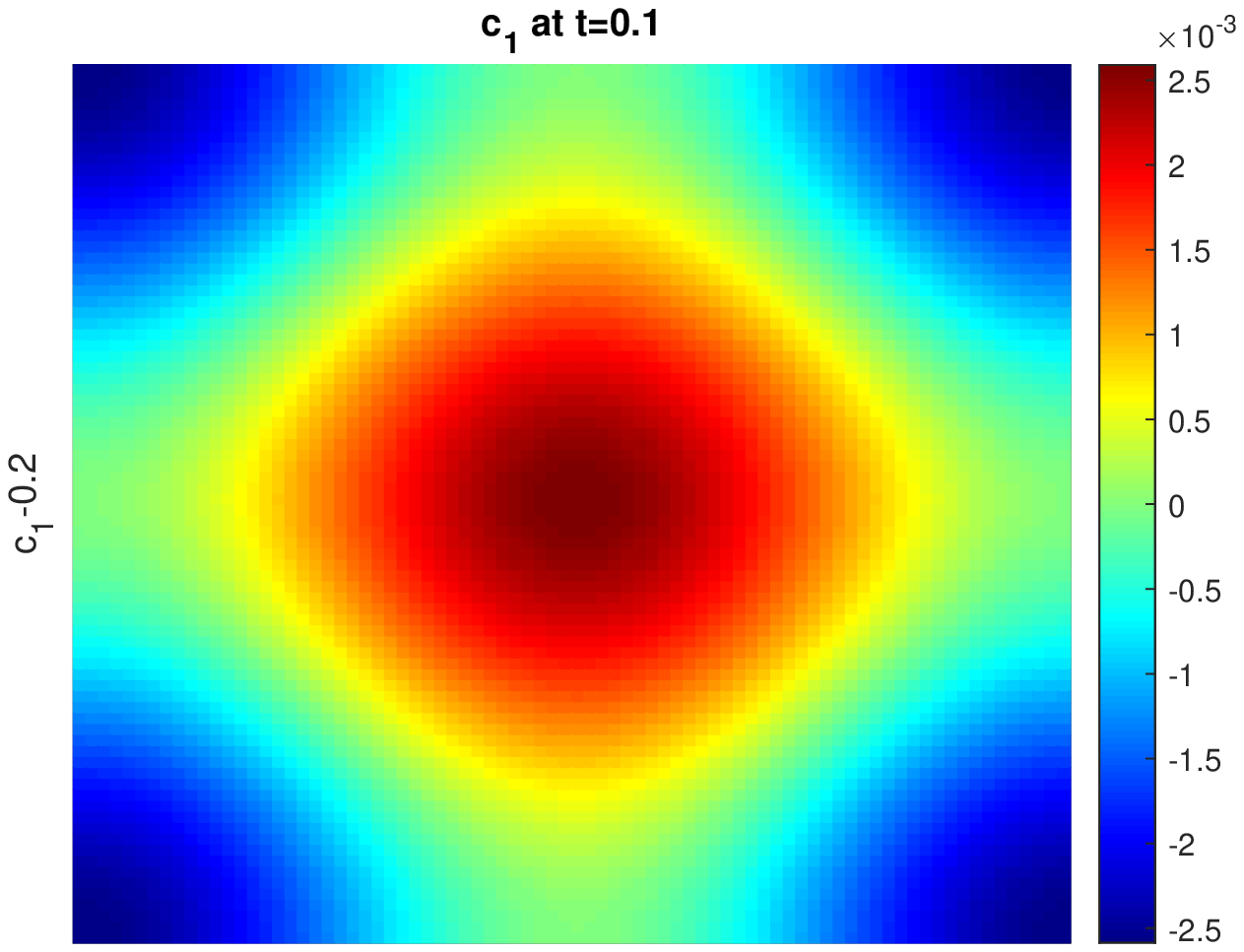}}
\subfigure{\includegraphics[width=0.325\textwidth]{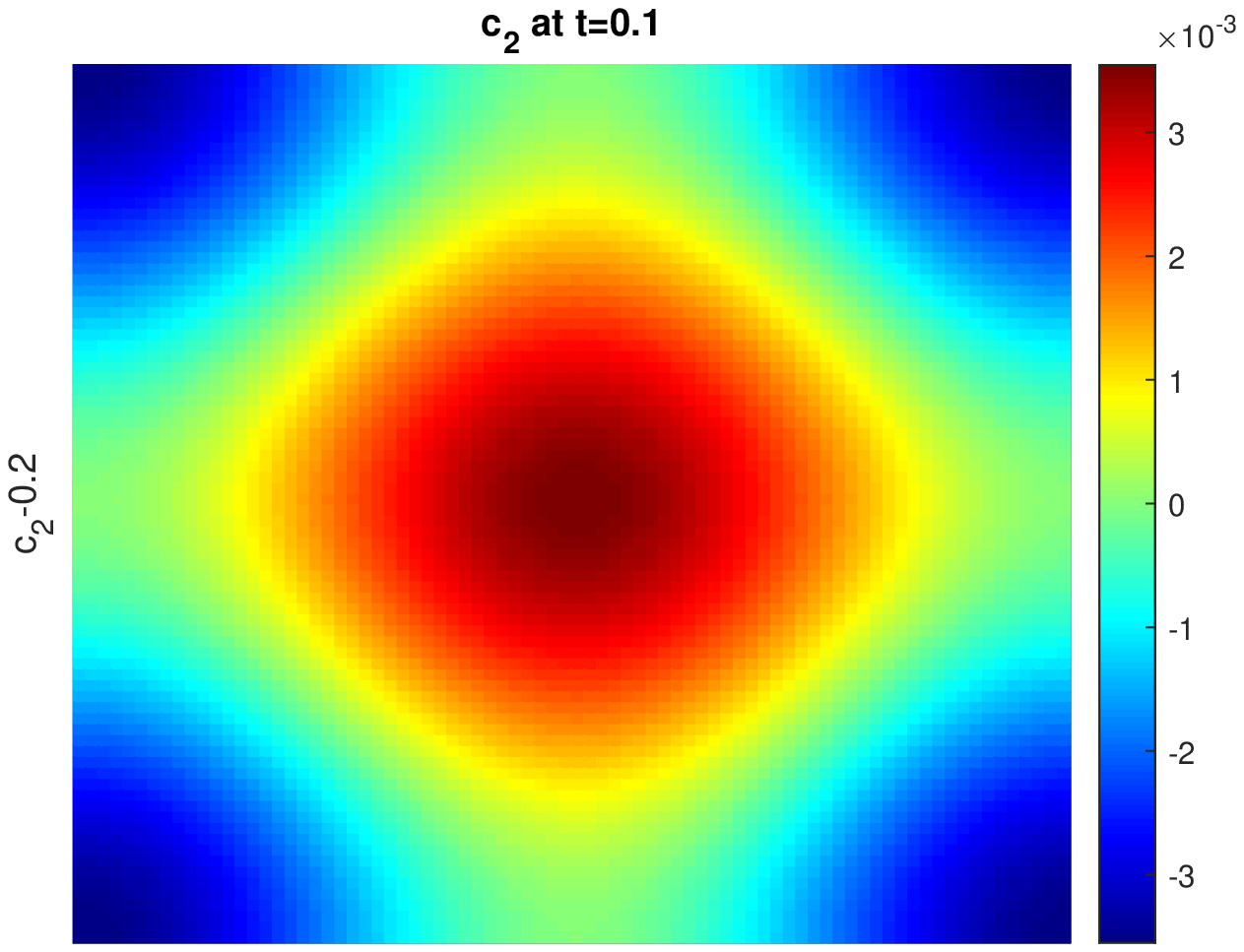}}
\subfigure{\includegraphics[width=0.325\textwidth]{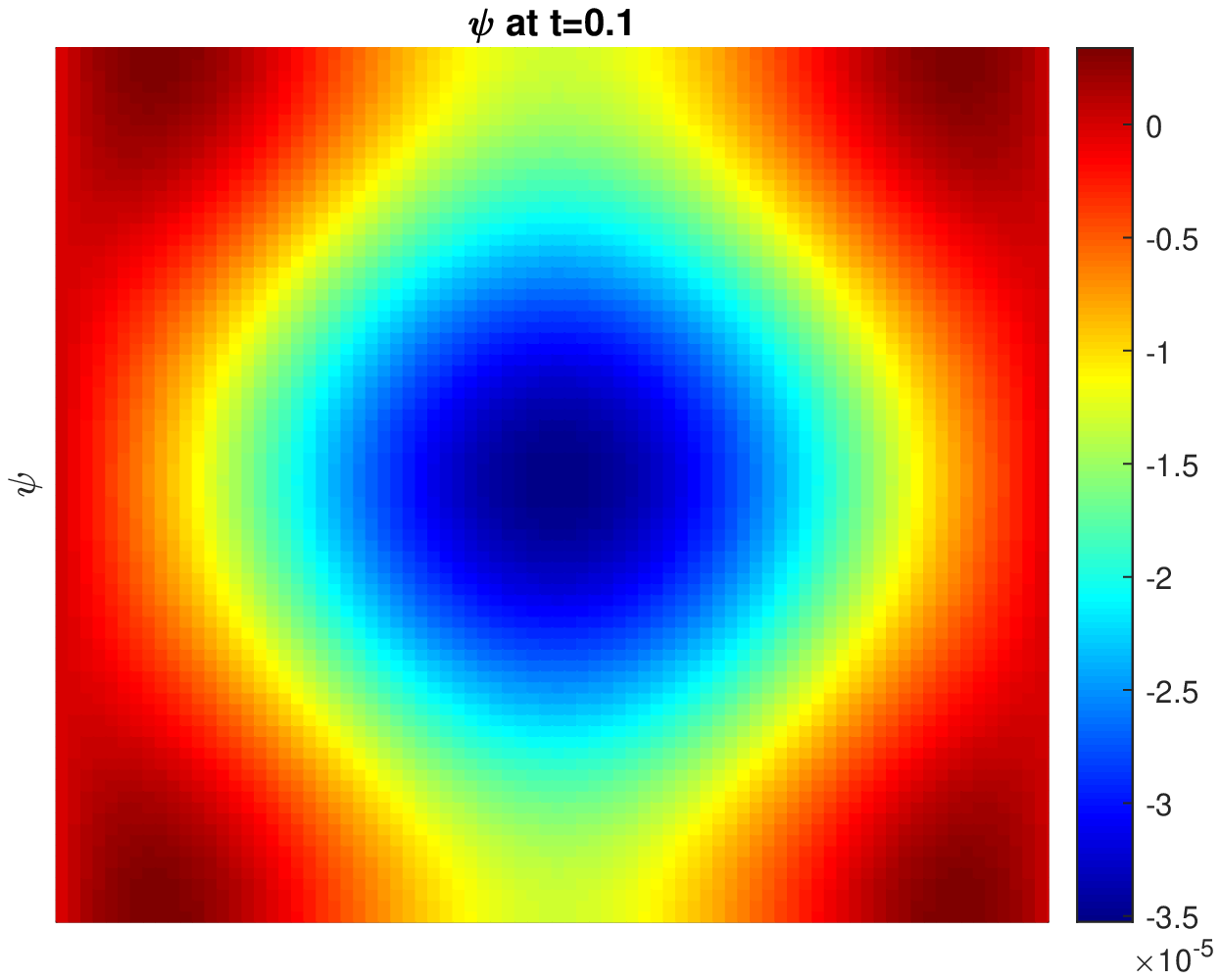}}
\subfigure{\includegraphics[width=0.325\textwidth]{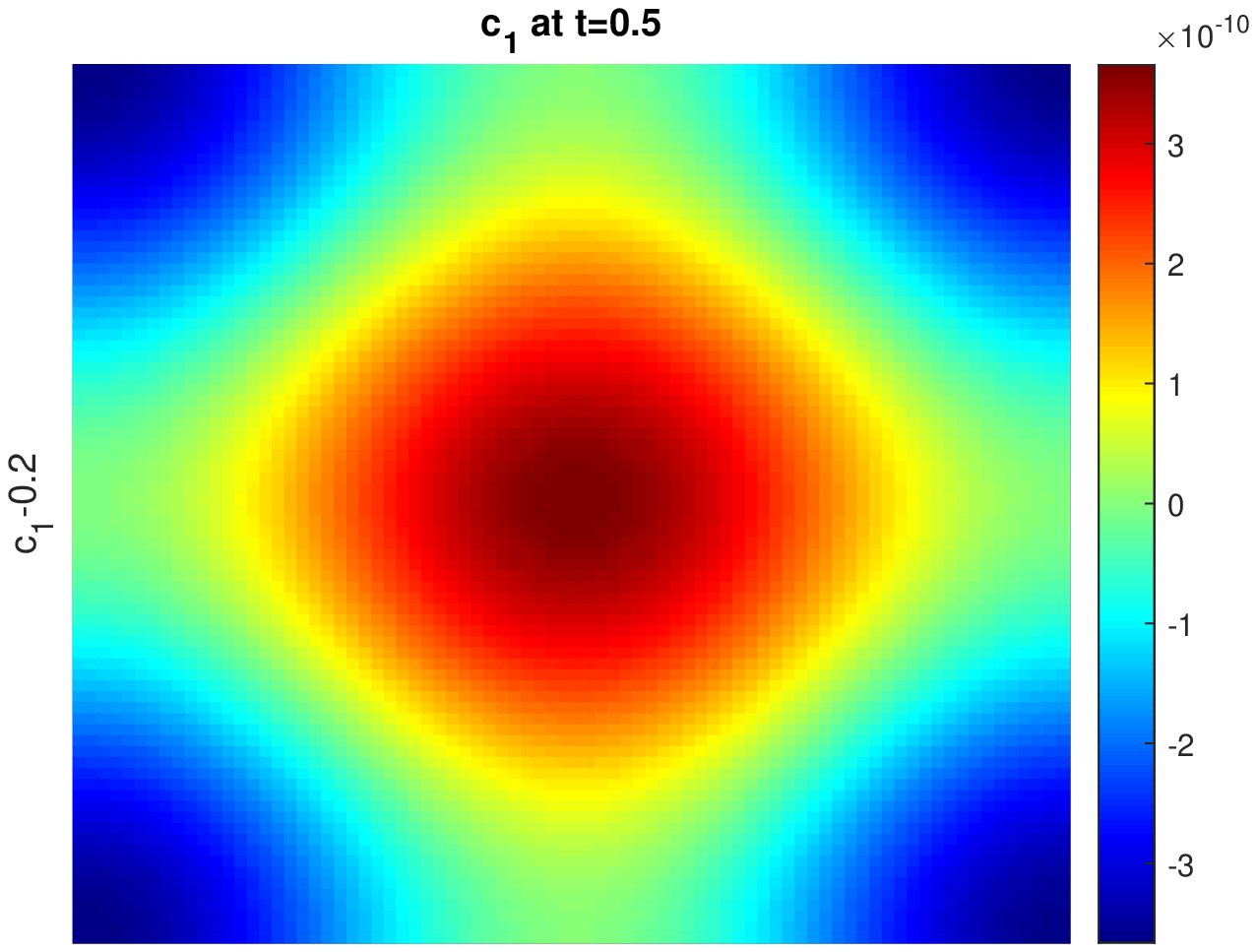}}
\subfigure{\includegraphics[width=0.325\textwidth]{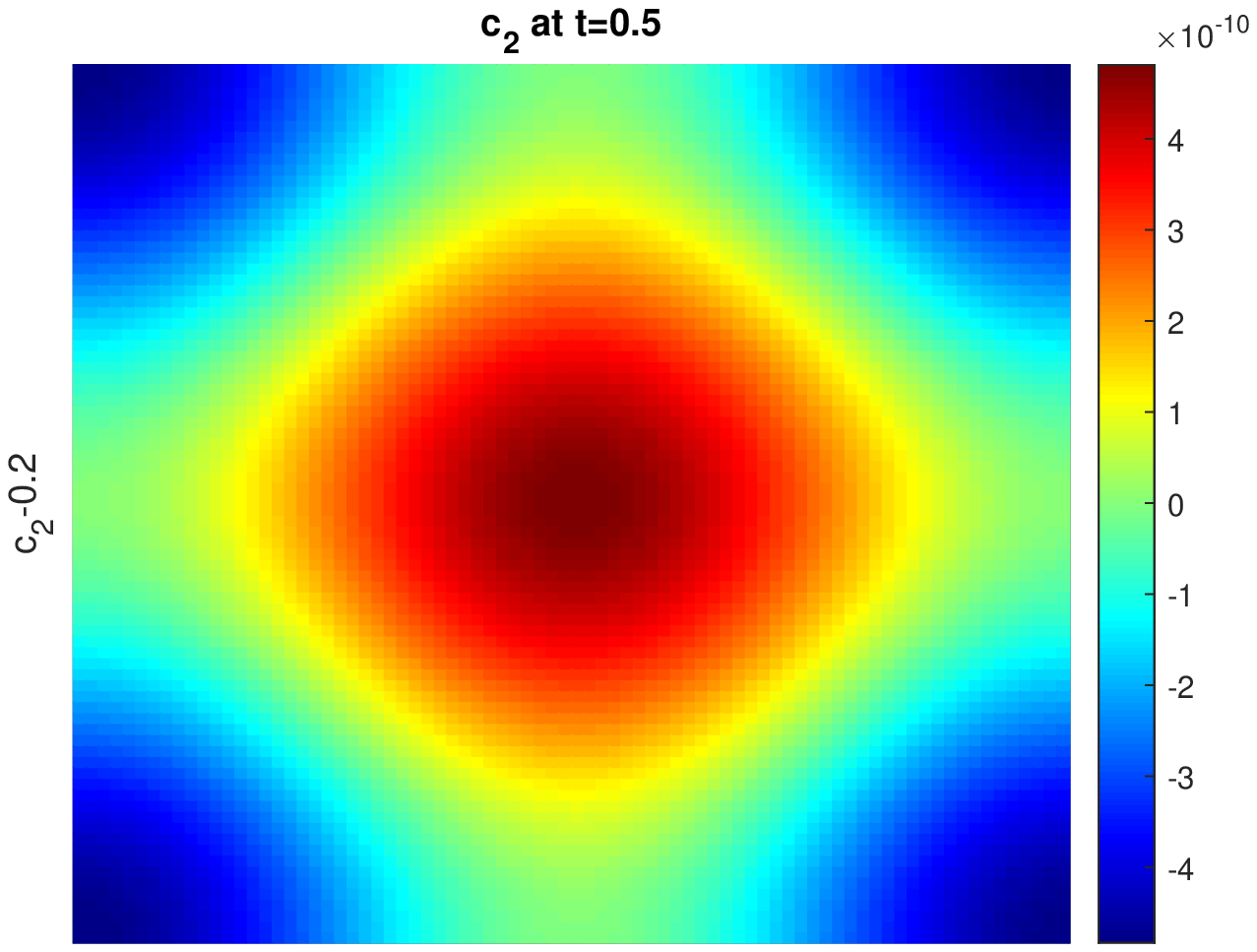}}
\subfigure{\includegraphics[width=0.325\textwidth]{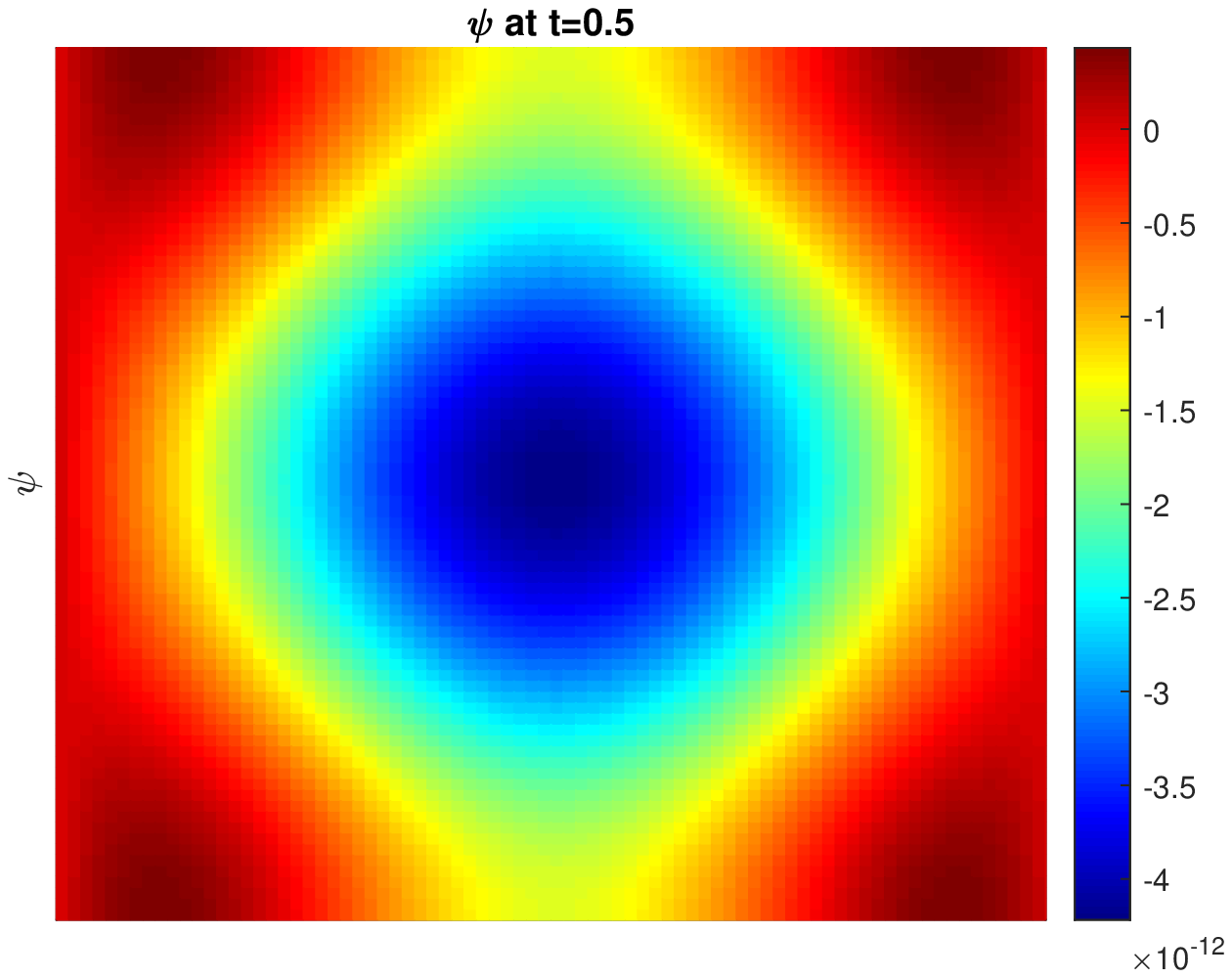}}
\caption{ The contours evolution of $c_1-0.2$, $c_2-0.2$ and $\psi$. } \label{ex5pattern}
\end{figure}

\begin{figure}[!htb]
\caption{Temporal evolution of the solutions}
\centering
\begin{tabular}{cc}
\includegraphics[width=\textwidth]{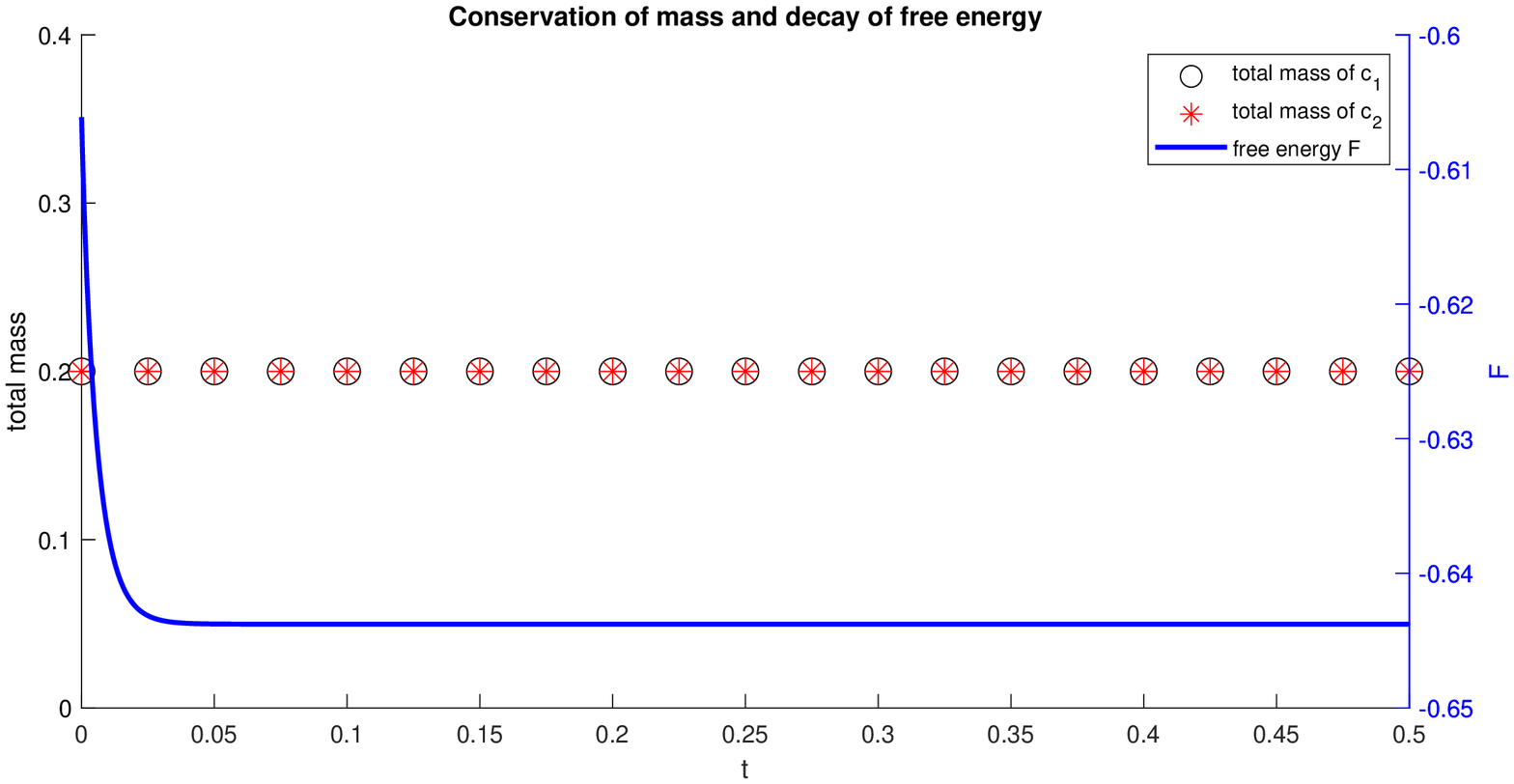} 
\end{tabular}
\label{fig:ex5}
\end{figure}

\section{Concluding remarks}
In this paper, we design and analyze third-order DDG schemes for solving time-dependent Poisson-Nernst-Planck systems, such equations are featured with non-negative density solutions. For admissible parameters in the DDG numerical flux, the weighted numerical integration allows for a positive decomposition over a test set of three points. As a result, with the Euler (or SSP-RK) time discretization and the positivity-preserving limiter, the fully discretized scheme is shown to preserve non-negativity of the numerical density. 
The schemes are also shown to conserve total mass for ion density when zero-flux boundary conditions are imposed, and preserve the steady states. Numerical examples are presented to demonstrate high resolution of the numerical algorithm and illustrate the proven property of positivity preserving and mass conservation, as well as the free energy decay.

\bigskip

\section*{Acknowledgments}  
Liu was partially supported by the National Science Foundation under Grant DMS1812666.


\end{document}